\documentclass[a4paper,10pt]{article}
\usepackage[left=3cm,right=3cm,top=3cm,bottom=3.5cm]{geometry}
\usepackage[latin1]{inputenc}
\usepackage{latexsym}
\usepackage{amscd}
\usepackage{amsthm}
\usepackage{amsfonts}
\usepackage{graphicx}
\usepackage{color}
\usepackage{multicol}
\usepackage{enumerate}
\usepackage{makeidx}
\usepackage{indentfirst}
\usepackage{subfigure}
\usepackage{a4wide}
\usepackage{t1enc}
\usepackage{amsmath,amssymb}
\usepackage{ae,aecompl}
\usepackage{booktabs}
\usepackage{verbatim}
\usepackage{lipsum}
\usepackage{capt-of}
\usepackage{authblk}

\usepackage[font=footnotesize,labelfont=bf]{caption}
\newcommand{\M}{\text{M}}
\newcommand{\m}{\text{m}}

\newcommand{\h}{\text{h}}

\newcommand{\J}{\text{J}}
\newcommand{\s}{\text{s}}
\newcommand{\C}{\text{C}}
\newcommand{\NH}{\text{NHIM}}
\newtheorem{theorem}{Theorem}
\newtheorem{proposition}[theorem]{Proposition}

\newtheorem{lemma}[theorem]{Lemma}

\theoremstyle{definition}
\newtheorem{remark}[theorem]{Remark}
\newtheorem{definition}[theorem]{Definition}
\begin{document}
\title{Arnold diffusion for a complete family of perturbations\footnote{This work has been partially supported by
      the Spanish MINECO-FEDER grant MTM2015-65715 and
      the Catalan grant 2014SGR504. AD has been also partially supported by
      the Russian Scientific Foundation grant 14-41-00044
      at the Lobachevsky University of Nizhny Novgorod. RS has been also partially supported by CNPq, Conselho Nacional de Desenvolvimento Cient\'{i}fico e Tecnol\'{o}gico - Brasil.}}
\author{Amadeu Delshams\thanks{amadeu.delshams@upc.edu}}
\author{Rodrigo G. Schaefer\thanks{rodrigo.schaefer@upc.edu}}  
\affil{Department de Matem\`atiques\\
   Universitat Polit\`ecnica de Catalunya\\
    Av. Diagonal 647, 08028 Barcelona}  
    
\maketitle

\begin{abstract}     
In this work we illustrate the Arnold diffusion in a concrete example---the \emph{a priori} unstable Hamiltonian system  of $2+1/2$ degrees of freedom  $H(p,q,I,\varphi,s) = p^{2}/2+\cos q -1 +I^{2}/2 + h(q,\varphi,s;\varepsilon)$---proving that for \emph{any} small periodic perturbation of the form
 $h(q,\varphi,s;\varepsilon) = \varepsilon\cos q\left( a_{00} + a_{10}\cos\varphi + a_{01}\cos s  \right)$ 
 ($a_{10}a_{01} \neq 0$) there is global instability for the action.  
For the proof we apply a geometrical mechanism based in the so-called Scattering map. 

This work has the following structure:
In a first stage, for a more restricted case ($I^*\thicksim\pi/2\mu$, $\mu = a_{10}/a_{01}$), we use only one scattering map, with a special property: the existence of simple paths of diffusion called highways.
Later, in the general case we combine a scattering map with the inner map (inner dynamics) to prove the more general result (the existence of the instability for any $\mu$). 
The bifurcations of the scattering map are also studied as a function of $\mu$. 
Finally, we give an estimate for the time of diffusion, and we show that this time is primarily the time spent under the scattering map.\\

\par\vspace{12pt}
  
  
\end{abstract}   

\section{Introduction}

The main goal of this paper is to describe the geometrical mechanism that gives rise to global instability in \emph{a priori unstable} Hamiltonians with $2 + 1/2$ degrees of freedom.
To such end we will consider the Hamiltonian
\begin{equation}\label{eq:hamiltonian_int}
H_{\varepsilon}(p,q,I,\varphi,s)=\pm\left(\frac{p^2}{2} + V(q) \right) + \frac{I^2}{2} +\varepsilon h(p,q,I,\varphi,s),
\end{equation}
where $p$, $I \in \mathbb{R}$, $q,\,\varphi,\,s\in\mathbb{T}$, with a potential $V$ and a perturbation $h$ given by
\begin{equation}\label{eq:perturbation_int}
V(q) = \cos q  - 1, \quad\quad h(p,q,I,\varphi,s) =\cos q\left( a_{00} + a_{10}\cos\varphi + a_{01}\cos s \right).
\end{equation}

\emph{A priori unstable} Hamiltonian systems like the above one were introduced by \cite{Chierchia1994,lochak}
They consist on a rotor in the variables $(I,\varphi)$ as an integrable Hamiltonian in action-angle variables, a pendulum in the variables $(p,q)$ which carries out a separatrix associated  to a saddle point, plus a small perturbation of size $\varepsilon$.
For $\varepsilon = 0 $ Hamiltonian \eqref{eq:hamiltonian_int} is integrable and, in particular, the action $I$ is constant.
We want to describe the global instability in the variable $I$ for $\left|\varepsilon\right|$ non-zero but otherwise arbitrary small.

For simplicity, we refer to global instability in this paper simply as Arnold diffusion.
Nevertheless, it is worth remarking that \emph{originally} the term Arnold diffusion was coined for a \emph{priori stable} Hamiltonian systems, which are perturbations of integrable Hamiltonian systems written in action-angle variable.
For instance, replacing $V(q)$ by $\varepsilon V(q)$, our Hamiltonian \eqref{eq:hamiltonian_int} becomes a priori stable.
In that case, Arnold diffusion would consisting on finding trajectories with large deviations $(p(T),I(T))- (p(0),I(0))$.
This would be a much more difficult problem that the one considered here, because one has to confront to exponentially small splitting of invariant manifolds with respect to the parameter $\varepsilon$ as well as to the passage through double res onances in the action variables $p,I$.
In particular, exponential large estimates of the time of diffusion with respect to $\varepsilon$ due to Nekhoroshev \cite{Nekhoroshev77,LochakM05,BounemouraM11} would apply. 

The main characteristic of an \emph{a priori unstable} Hamiltonian system with 2+$1/2$ degrees of freedom is that there exists a 3D Normally Hyperbolic Invariant Manifold (NHIM) which is a large invariant object with $4$D unstable and stable invariant manifolds.

Inside this NHIM there exists an inner dynamics given by a Hamiltonian system with $1+1/2$ degrees of freedom.
This Hamiltonian possesses  $2$D invariant tori which prevent global instability inside the 3D NHIM.

For $\varepsilon = 0$ the stable and unstable invariant manifold coincide  along a huge separatrix filled with homoclinic orbits to the NHIM. 

For small $\left|\varepsilon\right| \neq 0$, the unstable and stable manifolds of the NHIM in general do not coincide, but otherwise intersect transversely along 3D homoclinic invariant manifolds.
Through each point on each 3D homoclinic manifold, there exists a homoclinic orbit which begins in a point of the NHIM and finishes on another point of the NHIM, not necessarily the same one. 
This assignment between an initial and the final point on the NHIM is called the \emph{Scattering map}.
In practice, one must select an adequate domain for any scattering map.

Under the action of a scattering map, the variable $I$ can increase (or decrease).
The geometric mechanism of global instability consists on looking for trajectories of the scattering map with a large change on the variable $I$.
Standard shadowing arguments provide the existence of nearby trajectories of Hamiltonian \eqref{eq:hamiltonian_int} with a large change on the variable $I$.

Our first result is that the global instability happens for any arbitrary perturbation \eqref{eq:perturbation_int}.
\begin{theorem}\label{teo:main_theorem_int}
Consider the Hamiltonian \eqref{eq:hamiltonian_int} with the potential $V(q)$ and perturbation $h(p,q,I,\varphi,s)$ given in \eqref{eq:perturbation_int}.
Assume that $$a_{10}\,a_{01}\,\neq\,0$$

Then, for any $I^{*}>0$, there exists $0<\varepsilon^{*}=\varepsilon^{*}(I^{*},a_{10},a_{01})\ll 1$  such that for any $\varepsilon$,
$0 < \left|\varepsilon\right| < \varepsilon^{*}$, there exists a trajectory
$(p(t),q(t),I(t),\varphi(t))$ such
that for some $T>0$
$$I(0)\leq -I^*<I^*\leq I(T).$$
\end{theorem}
\begin{remark}
An upper bound for $\varepsilon^*$ can be estimated.
For instance, for $a_{10} = 0.6,\,a_{01}= 1$, given $I^* = 4$, it turns out $\varepsilon^* = 0.05$ is enough to guarantee global instability from $-4$ to $4$, see more details in sub-subsection \ref{gen_diff}.  
Alternatively, given $\varepsilon^*$ one can obtain a lower bound for the deviation $I(T) - I(0)$.
An expression for $T = T(\varepsilon^*,I^*,a_{10}, a_{01})$ is given in Theorem \ref{prop:time}. 
\end{remark} 

Let us mention that results about global instability are not new.
Indeed one can find related results in \cite{Seara2006,Delshams2009,lochak}.
Nevertheless, the main purpose of this paper is to describe the paths of instability that can be chosen as well as to estimate the time of diffusion.
In this sense, the choice of the simple model (1) will allow us:
\begin{enumerate}
\item To describe the map of heteroclinic orbits (scattering map) and to design fast paths of instability.

\item To describe bifurcations of the scattering maps as long as the parameter $\mu = a_{10}/a_{01}$ varies.

\item To estimate the time of diffusion along selected paths of instability. 
\end{enumerate}

To describe the scattering map let us recall how it can be computed.
To detect the intersections of the invariant manifolds associated to the NHIM one looks for non-degenerate critical points of the map
\begin{equation}\label{eq:critical_points_int}
\tau\rightarrow \mathcal{L}(I,\varphi - I\tau, s-\tau),
\end{equation}
where $\mathcal{L}(I,\varphi,s)$ is the so-called \emph{Melnikov potential}, which turns out to be for Hamiltonian \eqref{eq:hamiltonian_int} + \eqref{eq:perturbation_int}
\begin{equation*}
\mathcal{L}(I,\varphi,s)=A_{00}+A_{10}(I)\cos \varphi +A_{01}\cos s,
\end{equation*}
where 
\begin{equation*}
A_{00}=4\,a_{00}, \quad\quad  A_{10}(I)=\dfrac{2\,\pi\,I\,a_{10}}{\sinh(\frac{\pi I}{2})}, \quad\quad  A_{01}=\dfrac{2\,\pi\,a_{01}}{\sinh(\frac{\pi}{2})}.
\end{equation*}

Given $(I,\varphi,s)$ denote by $\tau^* = \tau^*(I,\varphi,s)$ one of the non-degenerate critical points of the function \eqref{eq:critical_points_int}, assuming that it exists.
Then the scattering map takes the form on the variables $(I,\theta = \varphi - Is)$:
\begin{equation}\label{eq:second_definition_SM_int}
\mathcal{S}(I,\theta)=\left(I+\varepsilon\,\frac{\partial \mathcal{L}^{*}}{\partial \theta}(I,\theta)+\mathcal{O}(\varepsilon^{2}),\theta-\varepsilon \, \frac{\partial \mathcal{L}^{*}}{\partial I}(I,\theta)+\mathcal{O}(\varepsilon^{2})\right),
\end{equation}
where  $\mathcal{L}^*(I,\theta) = \mathcal{L}(I,\theta -I\tau^*,-\tau^*)$ is the \emph{Reduced Poincar\'{e} function}.

Any different choice for a critical point $\tau^*$ gives rise to a different homoclinic manifold and to a different scattering map associated to it.
The location of the critical points $\tau^*(I,\varphi,s)$ of the function \eqref{eq:critical_points_int} in the torus $\{(\varphi,s)\in \mathbb{T}^2\}$ is therefore crucial for the definition and computation of the scattering map.

In Section \ref{sec:Inner and out} such critical points are determined by the value $\tau$ where the \NH\, lines $R_{\theta}(I)$ of equation 
\begin{equation*}
\varphi - Is = \theta
\end{equation*}
intersect the \emph{crests} $C(I)$ which  are given by the equation 
\begin{equation*}
\mu\alpha(I)\,\sin \varphi +\sin s =0,
\end{equation*}
where
\begin{equation*}
\alpha(I)=\frac{\sinh(\frac{\pi}{2})\,I^{2}}{\sinh(\frac{\pi\,I}{2})} \quad\quad \text{and}\quad\quad \mu=\dfrac{a_{10}}{a_{01}}. 
\end{equation*}

Subsection \ref{sub:Meln pot and crests} is devoted to describe the ``primary'' intersections between the \NH\, lines $R_{\theta}(I)$ and the crests $C(I)$ for all values of $\mu \neq 0,\infty$.
It turns out that there appear three different scenarios for the existence of scattering maps as a function of the bifurcation parameter $\mu$, as described in Theorem \ref{new_theorem}:
\begin{itemize}
\item For $0<\left|\mu\right|<0.625$, there exist two primary scattering maps defined on the whole range of $\theta\in \mathbb{T}$.

\item For $0.625<\left|\mu\right| <0.97$, there exist tangencies between the \NH\, lines and the crests giving rise to, at least, six scattering maps.

\item For $\left|\mu\right|>0.97$, for some bounded interval of $\left|I\right|$ there exists a sub-interval of $\theta$ in $\mathbb{T}$ such that the scattering maps are not defined.
\end{itemize}

By formula \eqref{eq:second_definition_SM_int} the trajectories of the scattering map are given by the -$\varepsilon$-time flow of the Hamiltonian $\mathcal{L}^*(I,\theta)$, up to order $\mathcal{O}(\varepsilon^2)$.
Therefore the phase space of the trajectories of the scattering map are well approximated by the level curves of the Reduced Poincar\'{e} function  $\mathcal{L}^*$, as long as the number of iterates is smaller than $1/\varepsilon$.

In Section \ref{sec:Arnold_diffusion} we display and study the geometric properties of the level curves of $\mathcal{L}^*$ and  we notice that there are some distinguished level sets of $\mathcal{L}^*$, namely $\mathcal{L}^*(I,\theta) = A_{00} + A_{01}$, called \emph{highways}, where the action $I$ increases or decreases very rapidly along close to vertical lines in the phase space $(\theta,I)$ (see Fig. \ref{fig:highways_example}).
Such highways are always defined for $\left|I\right|$ small (indeed, they are born on the inflection points of $ \mathcal{L}^*(0,\cdot)$) or $\left|I\right|$ large.

More precisely, in Proposition \ref{pro:geometrical_proposition} we see that for $\left|\mu\right| <0.9$, the highways are well defined for any value of $I$, whereas for $\left|\mu\right| >0.9$ they break along two intervals of $I$ ($\left|I\right| \in\left[I_{+}, I_{++}\right]$).

We finish this paper with an estimate of the time of diffusion, which for simplicity is presented only along the highways. 
Such estimate takes the form
\begin{equation*}
T_{\text{d}} = \frac{T_{\text{s}}}{\varepsilon}\log\left(\frac{C}{\varepsilon}\right) + \mathcal{O}(\varepsilon)\quad (\varepsilon \rightarrow 0).
\end{equation*}

Indeed, in Theorem \ref{prop:time}, we see that for selected diffusion trajectories, the diffusion time is basically given by the number of iterates of scattering maps, that is, the time under the inner map is negligible.
We notice that the form of this estimate agrees with the ``optimal'' estimates given by \cite{Berti2003,Tre04}, however we can provide concrete estimates for the constants $T_{\s}$ and $C$ as a function of $I^*,\,a_{10},\,a_{01}$, see Theorem \ref{prop:time}.

We finish the introduction by noting that all the results obtained with a perturbation \eqref{eq:perturbation_int} can be stated \emph{mutatis mutandis} for the following trivial generalization
$$h(p,q,I,\varphi,s) = \cos q \left(a_{00} + a_{10}\cos(k \varphi + l s) + a_{01}\cos s\right),\,k\neq 0,$$
since the change $\varphi' = k\varphi + l s$ gives our model with perturbation like \eqref{eq:perturbation_int} (with integrable Hamiltonian system for the inner dynamics).

Our results also apply for a more general perturbation like
$$h(p,q,I,\varphi,s) = \cos q \left(a_{00} + a_{10}\cos(k \varphi + l s) + a_{01}\cos(k'\varphi +l's)\right),\,\text{with}\,\begin{vmatrix}
 k & s \\ 
 k' & s'
\end{vmatrix} \neq 0,$$ 
although the concrete paths of diffusion needed require an additional description which is out of the scope of this paper.

\section{The system}\label{sec:the_system}

We consider the following \emph{a priori unstable} Hamiltonian with $2+1/2$ degrees of freedom with $2\pi$-periodic time dependence:
\begin{equation}\label{eq:hamil_system}
H_{\varepsilon}(p,q,I,\varphi,s)=\pm\left(\frac{p^2}{2} + \cos q - 1 \right) + \frac{I^2}{2} +\varepsilon f(q)g(\varphi,s),
\end{equation}
where $p$, $I \in \mathbb{R}$, $q,\,\varphi,\,s\in\mathbb{T}$ and $\varepsilon$ is 
small enough.

In the unperturbed case, that is, $\varepsilon = 0$, the Hamiltonian $H_{0}$ represents the standard pendulum plus a rotor:
$$H_{0}(p,q,I,\varphi,s) = \frac{p^2}{2} + \cos q - 1 +\frac{I^2}{2},$$
with associated equations
\begin{eqnarray}
\dot{q} = \frac{\partial H_{0}}{\partial p} = p &&\dot{p} = -\frac{\partial H_{0}}{\partial q} =\sin q\label{eq:qp}\\[2 mm] 
\dot{\varphi}=\frac{\partial H_{0}}{\partial I} = I& & \dot{I}=-\frac{\partial H_{0}}{\partial \varphi} =0  \nonumber \\
\dot{s}=1 & &  \nonumber
\end{eqnarray}
and associated flow  
$$\phi_t(p,q,I,\varphi,s) = \left(p(t),q(t),I,\varphi + It,s + t\right).$$
In this case, $(0,0)$ is a saddle point on the plane formed by variables $(p,q)$ with associated unstable and stable invariant curves.
Introducing $P(p,q)=p^2/2 + \cos q - 1$, we have that $P^{-1}(0)$ divides the $(p,q)$ phase space, separating the behavior of orbits.
The branches of $P^{-1}(0)$ are called \emph{separatrices} and are parameterized by the homoclinic trajectories to the saddle point $(p,q) = (0,0)$,
\begin{equation}
(p_{0}(t),q_{0}(t)) = \left(\frac{2}{\cosh t},4\arctan e^{\pm t}\right).\label{eq:separatrices}
\end{equation}
\begin{figure}[h]
\centering
\includegraphics[scale=0.27]{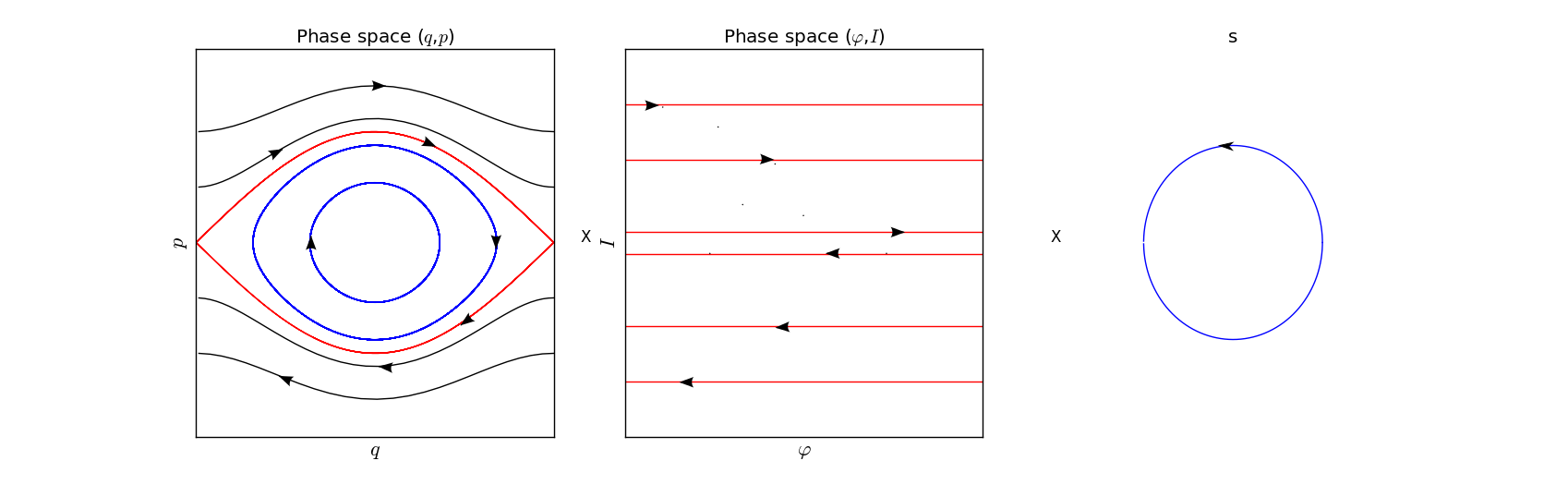}
\caption{ Phase Space  - Unperturbed problem}\label{fig:phase_space}
\end{figure}

For any initial condition $(0,0,I,\varphi,s)$, the unperturbed flow is $\phi_t(0,0,I,\varphi,s) = (0,0,I,\varphi + It, s + t)$, that is, the torus $\mathcal{T}^{0}_I = \{(0,0,I,\varphi,s);\, (\varphi,s)\in\mathbb{T}^2\}$ is an invariant set for the flow.
$\mathcal{T}^0_I$ is called \emph{whiskered torus}, and we call \emph{whiskers} its unstable and stable manifolds, which turn out to be coincident:
$$W^{0}\mathcal{T}^{0}_{I} = \{(p_{0}(\tau),q_{0}(\tau),I,\varphi,s);\tau\,\in\,\mathbb{R}, (\varphi,s)\,\in\,\mathbb{T}^{2})\}.$$

For any positive value $I^*$, consider the interval $\left[-I^*,I^*\right]$ and the cylinder formed by an uncountable family of tori
$$ \widetilde{\Lambda} = \{\mathcal{T}^0_{I}\}_{I\in\left[-I^*,I^*\right]} 
= \{(0,0,I,\varphi,s);I\in\left[-I^*,I^*\right],(\varphi,s)\in\mathbb{T}^2\}.$$
The set $\widetilde{\Lambda} $ is a 3D-normally hyperbolic invariant manifold (NHIM) with 4D-coincident stable and unstable invariant manifolds:
$$W^{0}\widetilde{\Lambda} = \left\{(p_{0}(\tau),q_{0}(\tau),I,\varphi,s);\tau\,\in\,\mathbb{R},\,I\,\in\,\left[-I^{*},I^{*}\right],\, (\varphi,s)\,\in\,\mathbb{T}^{2}\right\}.$$

We now come back to the perturbed case, that is, small $\left|\varepsilon\right|\neq 0$.
By the theory of NHIM (see for instance \cite{Seara2006} for more information), if $f(q)g(\varphi,s)$ is smooth enough, there exists a smooth NHIM $\widetilde{\Lambda}_{\varepsilon}$  close to $\widetilde{\Lambda}$ and the local invariant manifolds $W^{\text{u}}_{\text{loc}}(\widetilde{\Lambda}_ {\varepsilon})$ and $W_{\text{loc}}^{\text{s}}(\widetilde{\Lambda}_{\varepsilon})$ are $\varepsilon$-close to $W^{0}(\widetilde{\Lambda})$.
Indeed, 
$$W_{\text{loc}}^{\text{u,s}}(\widetilde{\Lambda}_{\varepsilon}) = \bigcup_{\tilde{x}\in\widetilde{\Lambda}_{\varepsilon}} W_{\text{loc}}^{\text{u,s}}(\tilde{x}),$$ 
where $W_{\text{loc}}^{\text{u,s}}(\tilde{x})$ are the unstable and stable manifolds associated to a point $\tilde{x}\in\widetilde{\Lambda}_{\varepsilon}$ (more precise information about the differentiability of $\widetilde{\Lambda}_{\varepsilon}$ and $W^{\text{u,s}}(\widetilde{\Lambda}_{\varepsilon})$ can be found in \cite{Seara2006}). 
Notice that if $f'(0) = 0$, $\tilde{\Lambda}_{\varepsilon} = \tilde{\Lambda}$, that is, $\tilde{\Lambda}$ is a NHIM for all $\varepsilon$. 
But even in this case, in general $W^{u}(\tilde{\Lambda}_ {\varepsilon})$ and $W^{s}(\tilde{\Lambda}_{\varepsilon})$ do not need to coincide, that is, the  \emph{separatrices} split.

Along this paper, we are going to take
\begin{equation}
f(q) = \cos q\quad\quad \text{and}\quad\quad g(\varphi,s) = a_{00} + a_{10}\cos\varphi + a_{01}\cos s ,\quad(a_{10}a_{01}\neq 0)\label{fun:f_and_g}
\end{equation}
so that there exists a normally hyperbolic invariant manifold $\tilde{\Lambda}_{\varepsilon} = \tilde{\Lambda}$ in the dynamics associated to the Hamiltonian \eqref{eq:hamiltonian_int}$+$\eqref{eq:perturbation_int}
\begin{equation*}
H_{\varepsilon}(p,q,I,\varphi,s)=\pm\left(\frac{p^2}{2} + \cos q - 1 \right) + \frac{I^2}{2} +\varepsilon \cos q\left( a_{00} + a_{10}\cos\varphi + a_{01}\cos s \right).
\end{equation*}

\begin{remark}
We are  choosing $f(q)$ as in \cite{Delshams2011} and a similar $g(\varphi,s)$.
Indeed, in \cite{Delshams2011}, $g(\varphi,s)=\sum_{(k,l) \in \mathbb{N}^{2}} a_{k,l}\cos(k\varphi - ls-\sigma_{k,l})$ is a \emph{full} trigonometrical series with the condition
\begin{equation*}\label{eq:aklb1}
{\hat\alpha} {\rho}^{\beta k} {r}^{\beta l} \leq |a_{k,l}| \leq \alpha
\rho^k r^l,
\end{equation*}
for $0<\rho \leq \rho^{*}$ and $0<r \leq r^{*}$, where $\rho^{*}(\lambda,\alpha, \hat \alpha, \beta)$ and $r^{*}(\lambda,\alpha, \hat \alpha, \beta)$ are small enough.
Under these hypothesis, the Melnikov potential, after ignoring terms of order greater or equal than 2, is the same Melnikov potential that we will obtain in the subsection \ref{sec:Melnikov_potential}. 
However, the inner dynamics in \cite{Delshams2011} is different.
In our case, as we will see, it is integrable,
therefore it is trivial and we will not worry about KAM theory to study the perturbed dynamics inside $\widetilde{\Lambda}_{\varepsilon}$.
\end{remark}

\section{The inner and the outer dynamics\label{sec:Inner and out}}

We have two dynamics associated to $\tilde{\Lambda}_{\varepsilon}$, the inner and the outer dynamics.
For the study of the inner dynamics we use the \emph{inner map} and for the outer one we use the \emph{scattering map}.
When it be convenient we will combine the scattering map and the inner dynamics to show the diffusion phenomenon.

\subsection{Inner map}

The inner dynamics is the dynamics in the NHIM.
Since $\tilde{\Lambda}_{\varepsilon} = \tilde{\Lambda}$, the Hamiltonian $H_{\varepsilon}$ restricted to $\tilde{\Lambda}_{\varepsilon}$ is
\begin{equation}
K(I,\varphi,s;\varepsilon)=\frac{I^{2}}{2}+\varepsilon\,\left(a_{00}+a_{10}\cos \varphi +a_{01}\cos s\right),\label{eq:inner_Hamil}
\end{equation}
with associated Hamiltonian equations
\begin{eqnarray}\label{eq.inner_equations}
\dot{\varphi} = I \quad &\quad \dot{I}=\varepsilon\,a_{10}\sin \varphi  \quad &\quad \dot{s}= 1.
\end{eqnarray} 

Note that the first two equations just depend of the variables $I$ and $\varphi$, thus using that 
$$F(I,\varphi):=\frac{I^{2}}{2}+\varepsilon\,a_{10}\left(\cos\varphi - 1\right) = K(I,\varphi,s) - \varepsilon\left(a_{00} +a_{10}\cos s -1\right)$$ 
is a first integral and indeed a Hamiltonian function for equations (\ref{eq.inner_equations}), one has that the inner Hamiltonian system (\ref{eq:inner_Hamil}) is integrable.
Therefore, here does not appear a genuine ``big gap problem'', and it does not require the KAM theorem to find invariant tori, since there is a continuous foliation of invariant tori simply given by $F=$ constant.
When $\varepsilon$ is small enough we have that the solutions are close to $I=$ constant, that is the level curves of $F$ are almost `flat' or `horizontal' in the action $I$ (see Fig. \ref{fig:inner_map}).

\begin{figure}[h]
\centering
\includegraphics[scale=0.27]{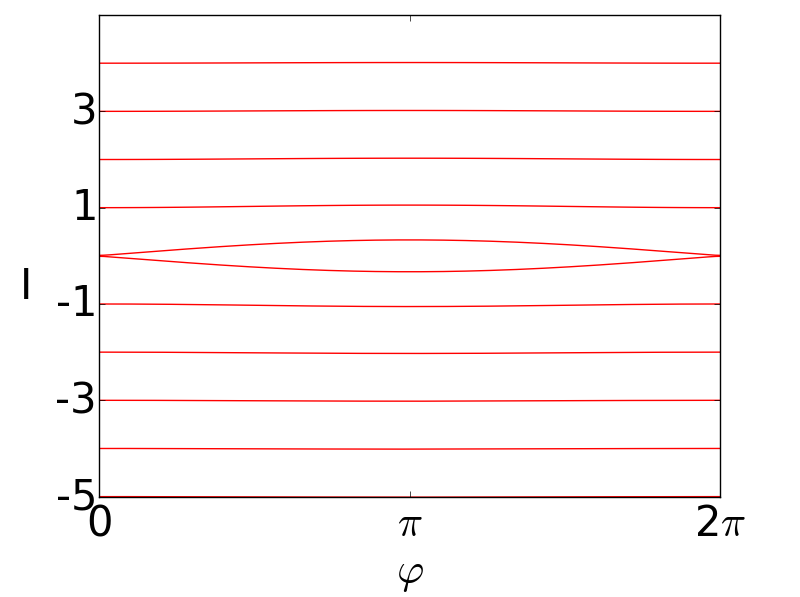}
\caption{Inner dynamics in the variables $(\varphi,I)$ for $a_{10} = 0.6$ and $\varepsilon = 0.01$ }\label{fig:inner_map}
\end{figure}
\subsection{Scattering map: Melnikov potential and crests\label{sub:Meln pot and crests}} 

The scattering map was introduced in \cite{Delshams2000} and is our main object of study.
Let $\tilde{\Lambda}$ be a NHIM with invariant manifolds intersecting transversally along a homoclinic manilfold $\Gamma$.
A scattering map is a map $S$ defined by $S(\tilde{x}_{-})=\tilde{x}_{+}$ if there exists $\tilde{z}\,\in\,\Gamma$ satisfying 
\begin{equation*}
\left|\phi_{t}(\tilde{z})-\phi_{t}(\tilde{x}_{\pm})\right|\,\longrightarrow\,0 \text{ as } t\,\longrightarrow\, \pm\infty 
 \end{equation*}
that is, $W_{\varepsilon}^u(\tilde{x}_-)$ intersects (transversally) $W^s_{\varepsilon}(\tilde{x}_+)$ in $\tilde{z}$.

For a more computational and geometrical definition of scattering map, we have to study the intersections between the hyperbolic invariant manifolds of $\tilde{\Lambda}_{\varepsilon}$.
We will use the Poincar\'{e}-Melnikov theory.

\subsubsection{Melnikov potential \label{sec:Melnikov_potential}}

We have the following proposition \cite{Delshams2011,Seara2006}.
\begin{proposition}\label{prop:melnpot}
Given $(I,\varphi,s)\,\in\,\left[-I^{*},I^{*}\right]\,\times\,\mathbb{T}^{2}$, assume that the real function 
\begin{equation}\label{eq: SM_critical_point}
\tau\,\in\,\mathbb{R}\,\longmapsto\,\mathcal{L}(I,\varphi-I\,\tau,s-\tau)\,\in\,\mathbb{R}
\end{equation}
has a non degenerate critical point $\tau^{*}\, =\, \tau^*(I,\varphi,s)$, where 
\begin{equation*}
\mathcal{L}(I,\varphi,s)=\int_{-\infty}^{+\infty}\left(f(q_{0}(\sigma))g(\varphi+I\sigma,s+\sigma;0)-f(0)g(\varphi+I\sigma,s+\sigma;0)\right)d\sigma.
\end{equation*}
Then, for $0\,<\,\left|\varepsilon\right|$ small enough, there exists a unique transversal homoclinic point $\tilde{z}$ to $\tilde{\Lambda}_{\varepsilon}$, which is $\varepsilon$-close to the point
$\tilde{z}^{*}(I,\varphi,s)\,=\,(p_{0}(\tau^{*}),q_{0}(\tau^{*}),I,\varphi,s)\,\in\,W^{0}(\tilde{\Lambda})$:
\begin{equation}
\tilde{z}=\tilde{z}(I,\varphi,s)=(p_{0}(\tau^{*})+O(\varepsilon), q_{0}(\tau^{*})+O(\varepsilon),I,\varphi,s)\,\in\,W^{u}(\tilde{\Lambda_ {\varepsilon}})\,\pitchfork\,W^{s}(\tilde{\Lambda_{\varepsilon}}).\label{eq:heteroclinic_point}
\end{equation}
\end{proposition}

The function $\mathcal{L}$ is called the \emph{Melnikov potential} of Hamiltonian \eqref{eq:hamil_system}.
In our case, from \eqref{eq:perturbation_int},\eqref{eq:separatrices} and \eqref{fun:f_and_g}
\begin{equation}
\mathcal{L}(I,\varphi,s)=A_{00}+A_{10}(I)\cos \varphi +A_{01}\cos s,\label{eq:our_meln_potential}
\end{equation}
where $A_{00}=4\,a_{00}$,
\begin{equation}
A_{10}(I)=\dfrac{2\,\pi\,I\,a_{10}}{\sinh(\frac{\pi\,I}{2})} \quad \text{ and } \quad A_{01}=\dfrac{2\,\pi\,a_{01}}{\sinh(\frac{\pi}{2})}.\label{eq:def_coef_mel_pot}
\end{equation}
\begin{figure}[h]
\centering
\includegraphics[scale=0.45]{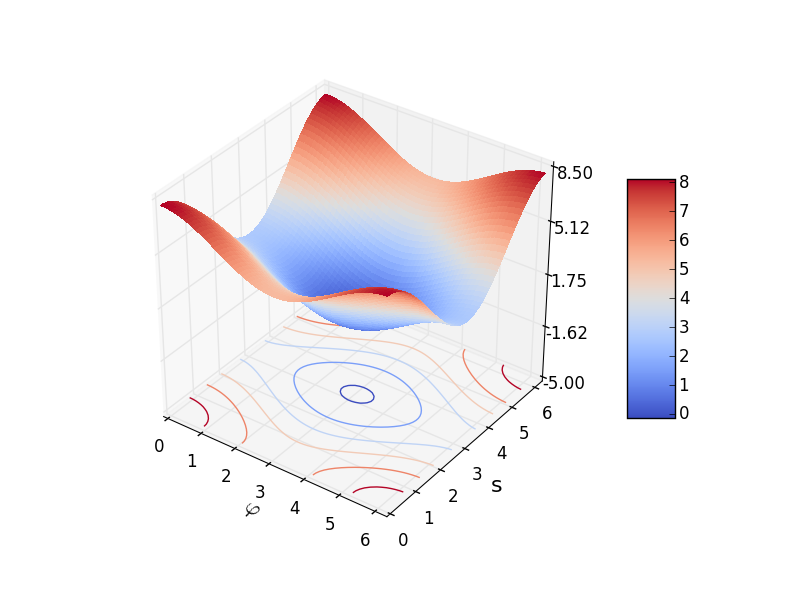}
\caption{The Melnikov potential, $\mu = a_{10}/a_{01} = 0.6 $ and $I = 1$.}
\end{figure}

We now look for the critical points of \eqref{eq: SM_critical_point} which indeed are the solutions of $\partial \mathcal{L}/\partial \tau(I,\varphi - I\tau,s-\tau) = 0$.  
Equivalently, $\tau^* = \tau^*(I,\varphi,s)$ satisfies 
\begin{equation}
I\,A_{10}(I)\sin(\varphi-I\,\tau^{*})+A_{10}\sin(s-\tau^{*})=0.\label{eq:def_of_tau}
\end{equation}

From a geometrical view-point,  for any $(I,\varphi,s)\in \left[-I^*,I^*\right]\times \mathbb{T}^2$ finding $\tau^* =\tau^*(I,\varphi,s)$ satisfying (\ref{eq:def_of_tau}) is equivalent to looking for the extrema of $\mathcal{L}$ on the \NH\, line 
\begin{equation}\label{eq:def_R} 
R(I,\varphi,s) =\{(I,\varphi-I\tau,s-\tau), \tau \in \mathbb{R}\},
\end{equation}
which correspond to the unperturbed trajectories of Hamiltonian $H_{0}$ along the unperturbed NHIM.

Thus we can define the \emph{scattering map} as in \cite{Delshams2011}.
Let $W$ be an open subset of $\left[-I^*,I^*\right]\times \mathbb{T}^2$ such that the map
$$(I,\varphi,s)\in W \mapsto \tau^*(I,\varphi,s),$$
where $\tau^*(I,\varphi,s)$ is a critical point of (\ref{eq: SM_critical_point}) or, equivalently, a solution of (\ref{eq:def_of_tau}), is well defined and $C^2$.
Therefore, there exists a unique $\tilde{z}$ satisfying (\ref{eq:heteroclinic_point}).
Let $\Gamma = \{\tilde{z}(I,\varphi,s;\varepsilon), (I,\varphi,s)\in W\}$.
For any $\tilde{z} \in \Gamma$ there exist unique $\tilde{x}_{+,-} = \tilde{x}_{+,-}(I,\varphi,s;\varepsilon)\in\tilde{\Lambda}_{\varepsilon}$ such that $\tilde{z} \in W^{s}_{\varepsilon}(\tilde{x}_-)\cap W^{u}_{\varepsilon}(\tilde{x}_+)$.
Let
$$H_{+,-} =\bigcup\{\tilde{x}_{+,-}(I,\varphi,s,\varepsilon),(I,\varphi,s)\in W\}.$$
We define the scattering map associated to $\Gamma$ as the map
\begin{eqnarray*}
S:H_-&\longrightarrow & H_+\\
\tilde{x}_-&\longmapsto & S(\tilde{x}_-)=\tilde{x}_+.
\end{eqnarray*} 

By the geometric properties of the scattering map (it is an exact symplectic map \cite{Delshams2008}) we have, see \cite{Delshams2009} and \cite{Delshams2011}, that the \emph{scattering map} has the explicit form 

\begin{equation}\label{eq:first_definition_SM}
S(I,\varphi,s) = \left(I+\varepsilon\,\frac{\partial L^{*}}{\partial \varphi}(I,\varphi,s)+\mathcal{O}(\varepsilon^{2}),\varphi-\varepsilon \, \frac{\partial L^*}{\partial I}(I,\varphi,s)+\mathcal{O}(\varepsilon^{2}),s\right),
\end{equation}
where
\begin{equation}
L^{*}(I,\varphi,s)=\mathcal{L}(I,\varphi-I\,\tau^{*}(I,\varphi,s),s-\tau^{*}(I,\varphi,s)).\label{eq:L^*}
\end{equation}

\paragraph{The new variable $\theta = \varphi - Is$\\\\}

Notice that if $\tau^*(I,\varphi,s)$ is a critical point of (\ref{eq: SM_critical_point}), $\tau^*(I,\varphi,s) -\sigma$ is a critical point of
\begin{equation}
\tau\longmapsto\mathcal{L}(I,\varphi- I(\tau + \sigma),s-(\tau + \sigma))=\mathcal{L}(I,\varphi - I\sigma - I\tau,s-\sigma-\tau).\label{eq:sigma}
\end{equation}

Since $\tau^*(I,\varphi-I\sigma,s-\sigma)$ is a critical point of the right hand side of (\ref{eq:sigma}), by the uniqueness in $W$ we can conclude that
\begin{equation}
\tau^*(I,\varphi - I\sigma,s-\sigma) = \tau^*(I,\varphi,s) - \sigma.\label{eq:prop_tau}
\end{equation}  
Thus, by \eqref{eq:L^*},
\begin{eqnarray*}
L^*(I,\varphi-I\sigma,s-\sigma) &= \mathcal{L}(I,\varphi-I\sigma -I(\tau^* - \sigma),s-\sigma-\tau^*)\\
&=\mathcal{L}(I,\varphi - I\tau^*,s-\tau^*)=L^*(I,\varphi,s),
\end{eqnarray*}
and, in particular for $\sigma = s$,
$$L^*(I,\varphi - Is,0) = L^*(I,\varphi,s).$$

Introducing the new variable
$$\theta = \varphi - Is,$$
we define \emph{the Reduced Poincar\'{e} function}
\begin{equation}
\mathcal{L}^*(I,\theta) := L^*(I,\varphi - Is,0) = L^*(I,\varphi,s).\label{eq:red_poi_func_1}
\end{equation}

We can write the scattering map on the variables $(I,\theta)$.
From $\left(I',\varphi',s'\right) = S(I,\varphi,s)$, we have that 
\begin{eqnarray*}
\theta' & =  \varphi'-I's'&= \left(\varphi - \varepsilon\frac{\partial L^{*}}{\partial I}(I,\varphi,s)\right) -\left(I + \varepsilon\frac{\partial L^{*}}{\partial\varphi}(I,\varphi,s)\right)s + \mathcal{O}(\varepsilon^2)\\
&&= \theta - \varepsilon\left(\frac{\partial L^{*}}{\partial I}(I,\varphi,s)+\frac{\partial L^{*}}{\partial\varphi}(I,\varphi,s)s \right) +\mathcal{O}(\varepsilon^2).
\end{eqnarray*}
Since 
\begin{equation*}
\frac{\partial L^*}{\partial I}(I,\varphi,s) =\frac{\partial \mathcal{L}^*}{\partial I}(I,\theta) - s\frac{\partial\mathcal{L}^*}{\partial \theta}(I,\theta) \quad\quad \text{ and }\quad\quad\frac{\partial L^*}{\partial \varphi}  = \frac{\partial \mathcal{L}^*}{\partial \theta}(I,\theta),
\end{equation*}
we conclude that
\begin{equation*}
\theta' = \theta - \varepsilon\left(\frac{\partial \mathcal{L}^*}{\partial I}(I,\theta)\right) + \mathcal{O}(\varepsilon^2) \quad\quad
\text{ and }\quad\quad
I' = I + \varepsilon\left(\frac{\partial \mathcal{L}^*}{\partial \theta}(I,\theta)\right) + \mathcal{O}(\varepsilon^2).
\end{equation*}

Then, in the variables $(I,\theta)$, the scattering map takes the simple form
\begin{equation}\label{eq:second_definition_SM}
\mathcal{S}(I,\theta)=\left(I+\varepsilon\,\frac{\partial \mathcal{L}^{*}}{\partial \theta}(I,\theta)+\mathcal{O}(\varepsilon^{2}),\theta-\varepsilon \, \frac{\partial \mathcal{L}^{*}}{\partial I}(I,\theta)+\mathcal{O}(\varepsilon^{2})\right),
\end{equation}
so up to $\mathcal{O}(\varepsilon^2)$ terms, $\mathcal{S}(I,\theta)$ is the $-\varepsilon$ times flow of the \emph{autonomous} Hamiltonian $\mathcal{L}^*(I,\theta)$.
In particular, the iterates under the scattering map follow the level curves of $\mathcal{L}^*$ up to $\mathcal{O}(\varepsilon^2)$.
\begin{remark}\label{rem:theta_quasi-periodic}
We notice that the variable $\theta$ is periodic in the variable $\varphi$ and quasi-periodic in the variable $s$.
Fixing $s$, then $\theta$ becomes periodic. 
\end{remark}

\begin{remark}\label{rem:norm_red_poin_fun}
Note that if for some values of $(I,\theta)$ we have that $\nabla\mathcal{L}^*(I,\theta) = \mathcal{O}(\varepsilon)$, so $\varepsilon\partial \mathcal{L}^{*}/\partial\theta(I,\theta)$ $=\mathcal{O}(\varepsilon^2) $ and $\varepsilon\partial \mathcal{L}^{*}/\partial I(I,\theta) = \mathcal{O}(\varepsilon^2)$. 
In this case, the level curves of $\mathcal{L}^*(I,\theta)$ do not provide the dominant part of the scattering map $\mathcal{S}$.
Therefore, we will be able to describe properly the scattering map through the level curves of the Reduced Poincar\'{e} function on the set of $(I,\theta)$ such that $\left\|\nabla \mathcal{L}^*(I,\theta)\right\|\gg\varepsilon$.
\end{remark}

\begin{remark} Using Eq.(\ref{eq:prop_tau}) and setting $s = \sigma$, we have that $\tau^*(I,\varphi - Is,0) =  \tau^*(I,\varphi,s) - s$. 
So we can define
\begin{equation}
\tau^*(I,\theta): = \tau^*(I,\varphi,s) - s\label{eq:tau_theta}
\end{equation}
and from (\ref{eq:L^*}) and (\ref{eq:red_poi_func_1}) we can write $\mathcal{L}^*$ as 
\begin{equation}
\mathcal{L}^*(I,\theta) = \mathcal{L}(I,\theta - I\tau^*(I,\theta),-\tau^*(I,\theta)).\label{eq:red_poi_fun_2}
\end{equation}
\end{remark}

\begin{remark}\label{rem:s_parameter}
In the variables $(I,\theta)$, the variable $s$ does \emph{not} appear at all in the expression \eqref{eq:second_definition_SM_int} for the scattering map, at least up to $\mathcal{O}(\varepsilon^2)$.
However, $s$ does appear in the expression \eqref{eq:first_definition_SM} in the original variables $(I,\varphi)$, so we have in \eqref{eq:first_definition_SM} a family of scattering maps parameterized by the variable $s$.
Playing with the parameter $s$, we can have scattering maps with different properties.
See Lemma \ref{lem:geometrical_lemmas} for an application of this phenomenon. 
\end{remark}

\subsubsection{The crests}

For the computation of the scattering maps, we use an important geometrical object introduced in \cite{Delshams2011}, the \emph{crests}. 
\begin{definition}
Fixed $I$, we define by \emph{crests} $\C(I)$ the curves on $(I,\varphi,s)$, $(\varphi,s)\in\mathbb{T}^2$, satisfying \begin{equation*}
I\frac{\partial \mathcal{L}}{\partial\varphi}(I,\varphi,s) + \frac{\partial \mathcal{L}}{\partial s}(I,\varphi,s)=0.
\end{equation*}
\end{definition}

In our case
\begin{equation}
I\,A_{10}(I)\sin \varphi +A_{01}\sin s =0. \label{eq:our_crests}
\end{equation} 

Note that a point $(I,\varphi,s)$ belongs to a crest $\C(I)$ if it is a minimum or maximum, or more generally, a critical point of $\mathcal{L}$ along a \NH\, line \eqref{eq:def_R}, that is, $\tau^*(I,\varphi,s) = 0$ in \eqref{eq:def_of_tau}, see Fig. \ref{fig:level_cur_crests}.

\begin{figure}[h]
\centering
\includegraphics[scale=0.27]{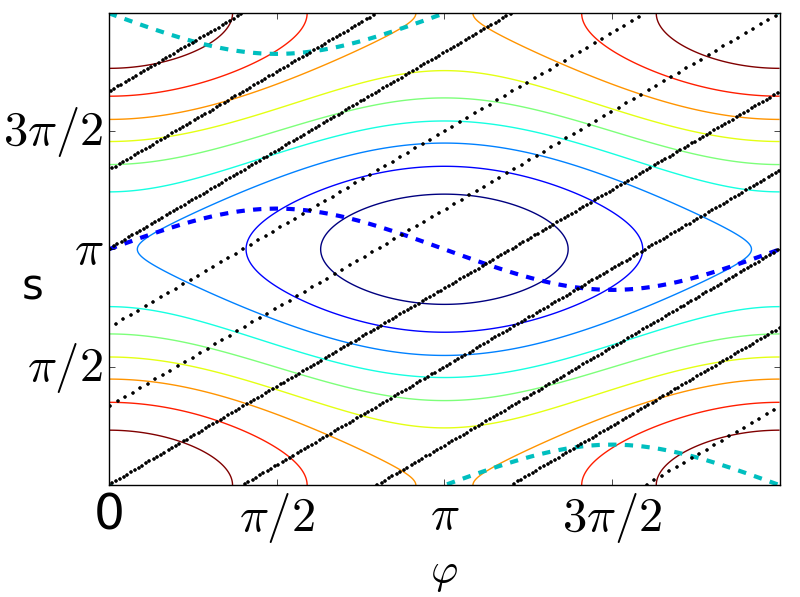}
\caption{Level curves of $\mathcal{L}$ for $\mu = a_{10}/a_{01} = 0.5$ and $I = 1.2$. Crests (dashed) in blue and green and the \NH\, lines in black. }\label{fig:level_cur_crests}
\end{figure}
 \begin{remark}
Note that any critical point of $\mathcal{L}(I,\cdot,\cdot)$ belongs to the crest $\C(I)$.
In general we have two curves satisfying Eq.(\ref{eq:our_crests}), the \emph{maximum} crest $\C_{\text{M}}(I)$, and the \emph{minimum} crest $\C_{\text{m}}(I)$.
The maximum crest contains the point $(I,\varphi = 0,s= 0)$, and the minimum crest the point $(I,\varphi = \pi,s =\pi)$.
For $a_{10}>0$, $a_{01}>0$, the Melnikov function has a maximum point at the point $(I,\varphi,s) = (I,0,0)$, and a minimum at $(I,\pi,\pi)$, and the function \eqref{eq: SM_critical_point} has a maximum on $\C_{\M}(I)$, and  a minimum on $\C_{\m}(I)$.
For other combinations of signs of $a_{10},\,a_{01}$, the location of maxima and minima changes, but for simplicity, we have preserved the name of maximum and minimum crest.   
\end{remark}

We now proceed to study the crests.
By \eqref{eq:def_coef_mel_pot} we can rewrite Eq. (\ref{eq:our_crests}) as
\begin{equation}
\mu\alpha(I)\,\sin \varphi +\sin s =0,\label{eq:def_crista}
\end{equation}
where
\begin{equation}
\alpha(I)=\frac{IA_{10}(I)}{\mu A_{10}}=\dfrac{\sinh(\frac{\pi}{2})\,I^{2}}{\sinh(\frac{\pi\,I}{2})}\quad\quad\text{ and }\quad\quad \mu=\dfrac{a_{10}}{a_{01}}.\label{eq:alpha}
\end{equation}

Note that if $\left|\mu\alpha(I)\right|<1$ we can write $s$ as a function of $\varphi$ for any value of $\varphi$.
On the other hand, if $\left|\mu\alpha(I)\right| >1$ we can write $\varphi$ as a function of $s$.
So, we have two different kinds of crests:
\begin{itemize}
\item For $\left|\alpha(I)\right|\, <\, 1/\left|\mu\right|$, the two crests are horizontal, see Fig. \ref{fig:horizontal_crests}, with 
$$\C_{\M,\m}(I)=\{(I,\varphi,\xi_{\M,\m}(I,\varphi)):\varphi\in\mathbb{T}\},$$
\begin{eqnarray}\label{eq:cristas_06}
\xi_{\M}(I,\varphi)&=&-\arcsin(\mu\alpha(I)\sin \varphi )  \quad\quad\mod{2\pi}\\
\xi_{\m}(I,\varphi)&=&\arcsin(\mu\alpha(I)\sin \varphi )+\pi \quad\quad\mod{2\pi}.\nonumber
\end{eqnarray}

\begin{figure}[h]
\centering
\subfigure[Horizontal crests: $\mu = a_{10}/a_{01} = 0.6$ and $I = 1.2$.]{\includegraphics[scale=0.27]{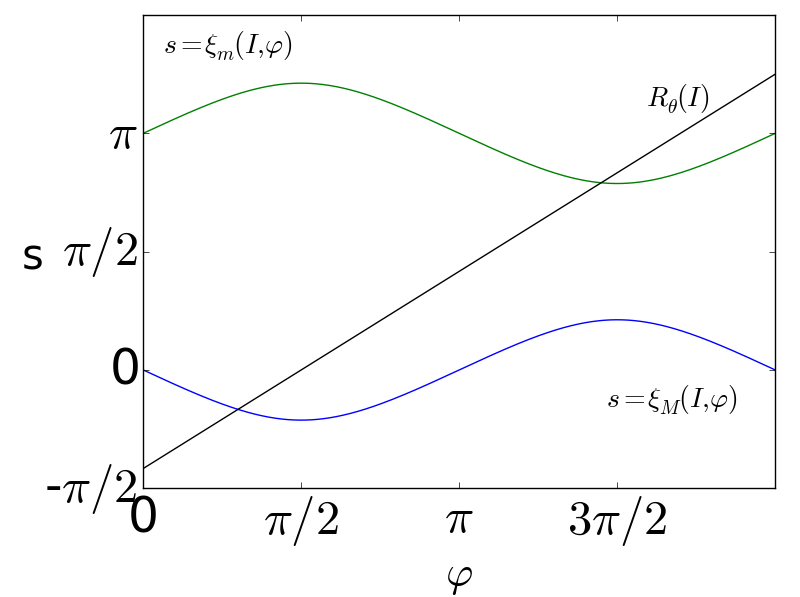}\label{fig:horizontal_crests}}
\qquad
\subfigure[Vertical crests: $\mu = a_{10}/a_{01} = 1.2$ and $I~ =~ 1$.]{\includegraphics[scale=0.27]{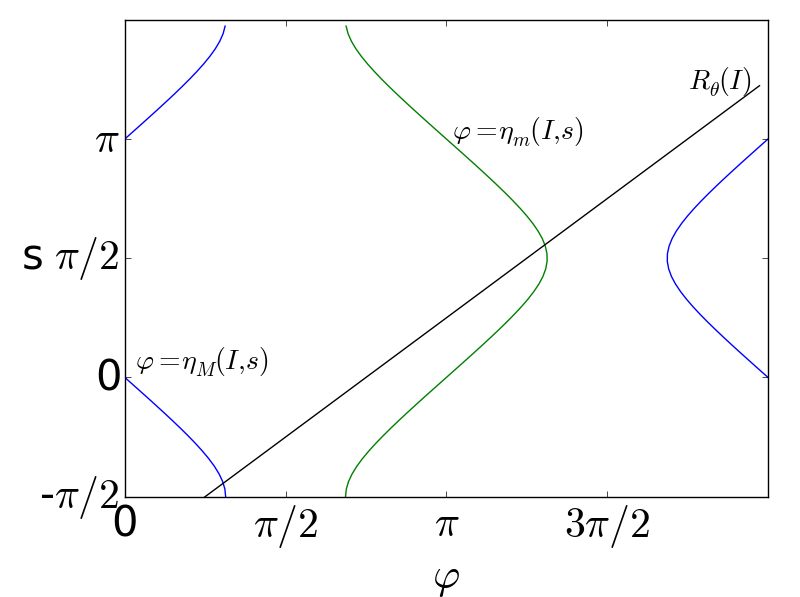}\label{fig:vertical_crests}}
\caption{Types of crests.}
\end{figure}
\item  For $\left|\alpha(I)\right|\, >\, 1/\left|\mu\right|$, the two crests are vertical, see Fig. \ref{fig:vertical_crests}, with
$$\C_{M,m}(I)=\{(I,\eta_{M,m}(I,s),s):s\in\mathbb{T}\},$$
\begin{eqnarray}
\eta_{M}(I,s)&=&-\arcsin(\sin s/\left(\mu \alpha(I)\right))\quad\quad\mod{2\pi} \label{eq:eta_definition}\\
\eta_{m}(I,s)&=&\arcsin(\sin s/\left(\mu \alpha(I)\right))+\pi\quad\quad\mod{2\pi}.\nonumber
\end{eqnarray}
\end{itemize}

\begin{remark}
The case $\left|\alpha(I)\right| = 1/\left|\mu\right|$ is singular, since both crests are piecewise \NH\, lines and they touch each other at the points $\left(\varphi,s\right) = \left(\pi/2,3\pi/2\right),\,\left(3\pi/2,\pi/2\right)$.
See Fig.~\ref{fig:singular_case}.
\begin{figure}[h]
\centering
\includegraphics[scale=0.25]{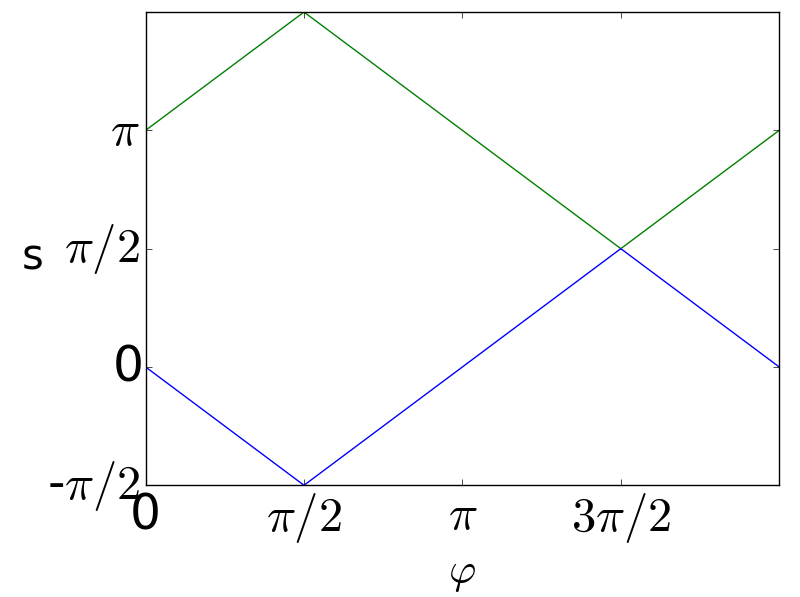}
\caption{Singular case: Crests for $I = 1$ and $\mu = 1$.}\label{fig:singular_case}
\end{figure}
\end{remark}
We can describe the relation between the crests $\C(I)$ and the \NH\, lines $R(I,\varphi,s)$ through the following Proposition:

\begin{proposition}\label{lem:crest}
Consider the crest $\C(I)$ defined by \eqref{eq:def_crista} and the \NH\, line $R(I,\varphi,s)$ defined in \eqref{eq:def_R}.
\begin{itemize}
\item[a)] For $\left|\mu\right| < 0.625$ the crests are horizontal and the intersections  between any crest and any \NH\, line is transversal.
\item[b)] For $0.625\leq\left|\mu\right| \leq 0.97$ the two crests $\C(I)$ are still horizontal, but for some values of $I$ there exist two \NH\, lines $R(I,\varphi,s)$ which are quadratically tangent to the crests.
\item[c)] For $\left|\mu\right| > 0.97$, the same properties as stated in b) hold, except that for $\left|\mu\alpha(I)\right|>1$, the crests $\C(I)$ are vertical.
\end{itemize}
\end{proposition}
\begin{proof}
The ``horizontality'' of a) and b) and the ``verticality'' of c) are due the upper bound of $\left|\mu\right|$.
Since $\left|\alpha(I)\right| < 1/0.97$ (see Fig.\ref{fig:module_alpha}), for $\left|\mu\right| \leq 0.97$, the crests are horizontal, that is, they can be expressed by equations \eqref{eq:cristas_06}. 

The condition of transversality is proved in \cite{Delshams2011}.
Essentially, the proof is to observe that  $\left|I\alpha(I)\right|<1.6$ and that there exists a $\varphi$ such that $\partial\xi(I,\varphi)/\partial \varphi = 1/I$ if, only if, $\left|I\alpha(I)\right|< 1/\left|\mu\right|$(we will prove it in a slightly different context, see the proof of Proposition \ref{pro:geometrical_proposition}.)

About the amount of \NH\, lines tangents to $\C(I)$, the proof is given in subsection \ref{sec:mult_scat_map}.
\end{proof}
\begin{figure}[h]
\centering
\includegraphics[scale=0.2]{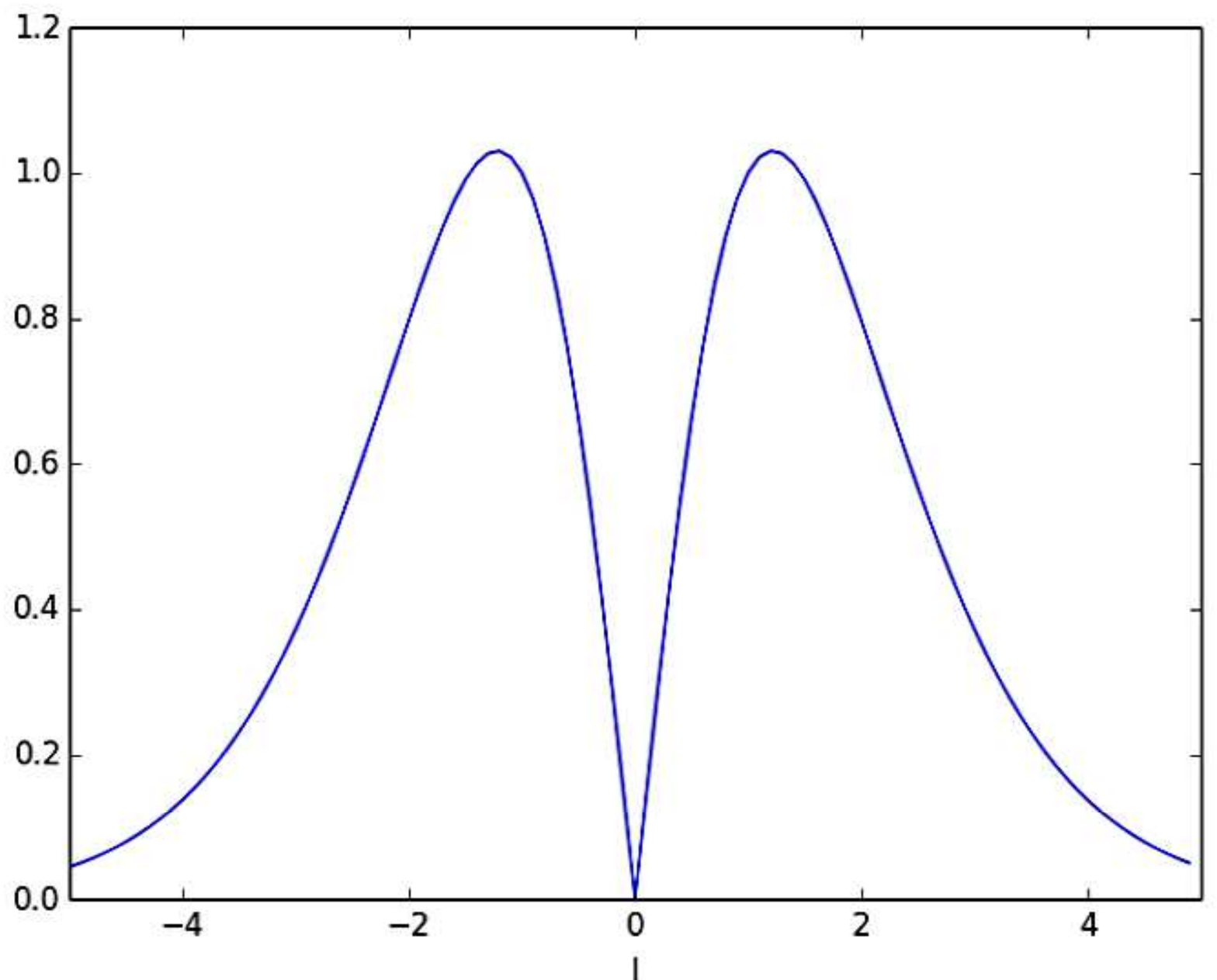}
\caption{Graph of $\left|\alpha(I)\right|$ }\label{fig:module_alpha}
\end{figure}
In Figs. \ref{fig:horizontal_crests} and \ref{fig:vertical_crests} we have displayed a segment of the the \NH\, line $R(I,\varphi,s)$, $\left|\tau\right| <\pi$, and we see that it intersects each crest $\C_{\text{M}}(I)$ and $\C_{\text{m}}(I)$  transversally, giving rise to two values $\tau^*_{\text{M}}$ and $\tau^*_{\text{m}}$ , therefore to two different scattering maps. 
We denote by $\tau^*_{\text{M}}$ the $\tau$ with minimum absolute value such that given $(I,\varphi,s)$, $(I,\varphi - I\tau,s-\tau) \in \C_{\text{M}}(I)$  and $\tau^*_{\text{m}}$ is defined analogously when $(I,\varphi - I\tau,s-\tau) \in \C_{\text{m}}(I)$ (see \cite{Delshams2011}).

\subsubsection{Scattering maps and crests}

Note that $\tau^{*}_{\text{m}}$ and $\tau^*_{\text{M}}$ are associated to different homoclinic points to the NHIM $\tilde{\Lambda}$, and consequently, to different homoclinic connections.
From this we build different scattering maps.
The most natural way is to associate one scattering map to each crest.
And we will do this on the variables $(I,\varphi,s)$ and $(I,\theta)$, where $\theta = \varphi - Is$.

Before, we make some considerations about the \NH\,lines defined in \eqref{eq:def_R}.
Note that 
$$\theta :=\varphi - Is = \left(\varphi - I\tau\right) - I\left(s-\tau\right),$$
that is, $\theta$ is constant on each \NH\, line $R(I,\varphi,s)$, so we will also introduce another notation for a \NH\,line $R(I,\varphi,s)$, namely
$$R_{\theta}(I):=\{(I,\varphi,s):\varphi-Is =\theta\}.$$

Since $(\varphi,s)\in\mathbb{T}^2$, $R(I,\varphi,s)$ is a closed line if $I\in\mathbb{Q}$, whereas it is a dense line on $\mathbb{T}^2$ if $I\notin\mathbb{Q}$.
In this case, $R(I,\varphi,s)$ intersects the crests $\C(I)$ on an infinite number of points.

Recall (see Remark \ref{rem:theta_quasi-periodic}) that $\theta$ is quasi-periodic in the variable $s\in\mathbb{T}$.
To avoid monodromy  with respect to this variable, we are going to consider from now on $s$ as a real variable in an interval of length $2\pi$, $-\pi/2 < s \leq 3\pi/2$.
Under this restriction, the \NH\, line $R(I,\varphi,s)$ defined in (\ref{eq:def_R}) becomes a \emph{\NH\,segment}
\begin{equation}
R(I,\varphi,s) = \{\left(I,\varphi-I\tau, s - \tau\right);-\pi/2< s-\tau\leq 3\pi/2\},\label{eq:line_segments_I_varphi}
\end{equation}
as well as $R_{\theta}(I)$, which can be written as
\begin{equation}
R_{\theta}(I) = \{(I,\varphi,s):\varphi -Is = \theta,(\varphi,s)\in\mathbb{T}\times\left(-\pi/2,3\pi/2\right]\}.\label{eq:line_segments_I_theta}
\end{equation}
From now on, when we refer to $R(I,\varphi,s)$ and $R_{\theta}(I)$, they will be these line segments.
Notice that $\theta \in \mathbb{T}$.

We begin to consider the \emph{primary} scattering map $\mathcal{S}_{\M}$ associated to the maximum crest $\C_{\M}$, that is, we look only at the intersections between the segment $R(I,\varphi,s)$ given in \eqref{eq:line_segments_I_varphi} and $\C_{\M}(I)$, parameterized by $\tau_ {\text{M}}^*(I,\varphi,s) = \tau_{\text{M}}^*(I,\theta) + s$ (see \eqref{eq:tau_theta}):
\begin{eqnarray}
\C_{\M}(I)\cap R(I,\varphi,s) & = & \left\{\left(I,\varphi - I\tau_{\M}^*(I,\varphi,s),\xi_{\M}(I,\varphi-I\tau_{\M}^*(I,\varphi,s))\right)\right\}\label{eq:inters_ret_crst_1}\\
&=&\left\{\left(I,\varphi - I\tau_{\M}^*(I,\varphi,s),s-\tau_{M}^*(I,\varphi,s)\right)\right\}\label{eq:inters_ter_crst_2}
\end{eqnarray}

Equation \eqref{eq:inters_ret_crst_1} motivates us to introduce a new variable $\psi = \varphi - I\tau_{\M}^*(I,\varphi,s)$ that will be useful in many contexts.

\pagebreak[4]

\paragraph{The variable $\psi$: a variable on the crest.\\\\}

Let $\C(I)$ be a crest such that it can be parameterized by $\xi(I,\varphi)$ as in (\ref{eq:cristas_06}).
Since $\tau^*(I,\varphi,s)$ is the value of $\tau$ such that $R(I,\varphi,s)$, given in (\ref{eq:line_segments_I_varphi}), intersects $\C(I)$, we define
\begin{equation}
\psi: = \varphi - I\tau^*(I,\varphi,s).\label{def:psi_original}
\end{equation}

By (\ref{eq:tau_theta}) we can  also write $\psi$ in terms of the variable $\theta$:
\begin{equation}\label{eq:def_psi}
\psi = \varphi - I\left(\tau^*(I,\theta) + s\right) = \theta - I\tau^*(I,\theta).
\end{equation}

By (\ref{eq:inters_ret_crst_1}) and (\ref{eq:inters_ter_crst_2}),
\begin{equation}
s-\tau^*(I,\varphi,s) = \xi(I,\varphi - I\tau^*(I,\varphi,s)),\label{eq:rel_bet_s-tau_and_xi}
\end{equation}
that is, $s-\tau^*(I,\varphi,s) = \xi(I,\psi)$.
In particular, for $s = 0$, $\xi(I,\psi) = -\tau^*(I,\varphi,0) = -\tau^*(I,\theta)$ again by \eqref{eq:tau_theta} and from \eqref{eq:def_psi} we have the expression of $\theta$ in terms of $\psi$:
\begin{equation}\label{eq:def_theta_psi}
\theta = \psi - I\xi(I,\psi).
\end{equation}
All the relations between the variables $(\varphi,s)$, $\theta$ and $\psi$ are written in Table \ref{tab:table_variables} and are displayed in Fig. \ref{fig:variables}.
By the definitions of $L^*(I,\varphi,s)$ in (\ref{eq:L^*}), and $\mathcal{L}^*(I,\theta)$ in \eqref{eq:red_poi_func_1} and \eqref{eq:red_poi_fun_2}, we have that
\begin{equation}
\mathcal{L}^*(I,\theta) = L^*(I,\varphi,s) = \mathcal{L}(I,\psi,\xi(I,\psi)),\label{eq:three_rpf}
\end{equation}
So we can define the reduced Poincar\'{e} function in terms of $(I,\psi)$ as
\begin{equation}
\mathfrak{L}^*(I,\psi):=\mathcal{L}(I,\psi,\xi(I,\psi)),\label{eq:red_poi_fun_psi}
\end{equation}
which in our case takes the simple and computable form
\begin{equation}\label{eq:red_poin_fun_psi_2}
\mathfrak{L}^*(I,\psi) = A_{00}+A_{10}(I)\cos\psi + A_{01}\cos\xi(I,\psi),
\end{equation}
for a horizontal crest \eqref{eq:cristas_06}.

Therefore, as $(I,\psi,\xi(I,\psi))$ are points on the crest, the domain of $L^*(I,\cdot,\cdot)$ is a subset of $C(I)$.
So, if there exist different subsets where $L^*(I,\cdot,\cdot)$ can be well defined, we can build different scattering maps associated to $C(I)$.

\begin{figure}[h]
\parbox[t]{8cm}{\null
  \centering
  \includegraphics[scale=0.4]{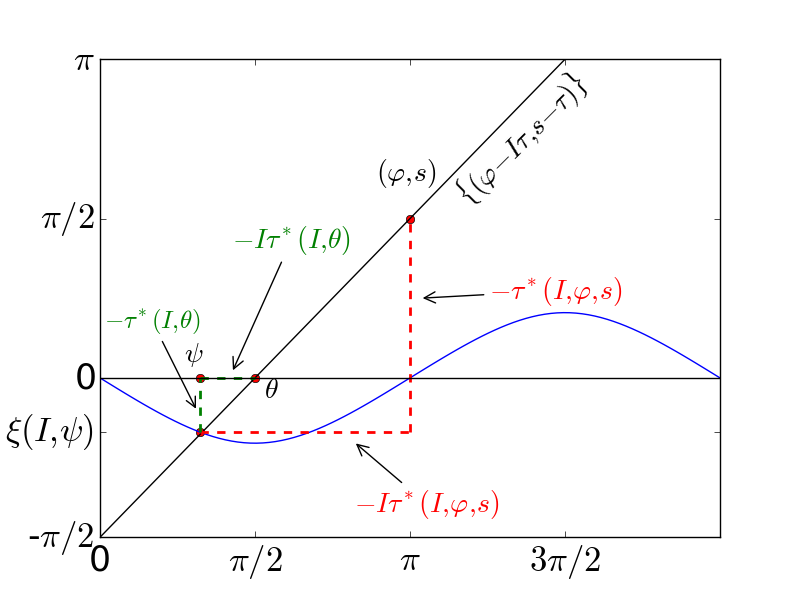} 
      \captionof{figure}{The three variables: $\varphi,\,\theta$ and $\psi$.}\label{fig:variables}%
}
\parbox[t]{7cm}{\null
\centering
\vspace{1cm}
{\footnotesize
\begin{tabular}{@{}|c|c|@{}}
\toprule
$\theta = \psi - I\xi(I,\psi)$ &  $\psi = \theta - I\tau^*(I,\theta)$ \\ \midrule
$\theta = \varphi - Is$ & $ \varphi =  \theta + Is$\\ \midrule
 $\psi = \varphi - I\tau^*(I,\varphi,s)$&$\varphi = \psi + I\left(s - \xi(I,\psi)\right)$  \\ \bottomrule
  \end{tabular}
  }
  \captionof{table}[t]{Relation between variables.}\label{tab:table_variables}%
}
\end{figure}

Denote $\mathcal{L}_i^*(I,\theta) = L(I,\varphi - I\tau^*_i(I,\varphi,s),s-\tau_i^*(I,\varphi,s))$, $i = \m,\M$, and $\mathfrak{L}^*_i(I,\psi)=\mathcal{L}(I,\psi,\xi_i(I,\psi))$ from \eqref{eq:three_rpf} and \eqref{eq:red_poi_fun_psi}.
We state the following lemma
\begin{lemma}\label{lem:geometrical_lemmas}
\begin{itemize}
\item[a)]The Poincar\'{e} Reduced functions $\mathfrak{L}^*_{\text{M}}(I,\psi)$ and $\mathcal{L}^*_{\M}(I,\theta)$  are even functions in the variable $I$, that is, $\mathfrak{L}^*_{\text{M}}(I,\psi) = \mathfrak{L}_{\text{M}}^*(-I,\psi)$ and $\mathcal{L}^*_{\M}(I, \theta) = \mathcal{L}_{\M}^*(-I,\theta)$, and consequently $\mathcal{S}_{\text{M}}(I,\theta)$ is symmetric in this variable $I$.
The same happens for $\mathcal{S}_{\text{m}}(I,\theta)$, that is, for the scattering map associated to $C_{\m}(I)$.
\item[b)]The scattering map for a value of $\mu$ and $s = \pi$, associated to the intersection between $R_{\theta}(I)$ and $C_{\m}(I)$ has the same geometrical properties as the scattering map for $-\mu$ and $s = 0$, associated to the intersection between $R_{\theta}(I)$ and $C_{\M}(I)$, i.e.,
\begin{equation*}\label{eq:equivalence_sign_mu}
S_{\mu,\text{m}}(I,\varphi,\pi) = S_{-\mu,\text{M}}(I,\varphi,0) = \mathcal{S}_{-\mu,\M}(I,\theta)
\end{equation*}
\end{itemize}
\end{lemma}

\begin{proof}
\begin{itemize}
\item[a)]This is an immediate consequence of the fact that function $A_{10}(I)$ is even and $\xi_{\M}(I,\varphi)$ is odd in the variable $I$, see  \eqref{eq:def_coef_mel_pot} and \eqref{eq:cristas_06}. 
\item[b)]First, we look for $\tau^*_{\text{m}}$ such that the \NH\, segment $R_\theta(I)$ intersects the crest $C_{\text{m}}(I)$.
If we fix $s=\pi$, we have by \eqref{eq:L^*} and \eqref{eq:our_meln_potential}:
\begin{equation}
L_{\mu,\m}^*(I,\varphi,\pi) = A_{00}+ A_{10}(I)\cos(\varphi-I\tau_{\text{m}}^*(I,\varphi,\pi))+A_{01}\cos(\pi-\tau_{\text{m}}^*(I,\varphi,\pi)).\label{eq:mel_pot_phi_mu_neg}
\end{equation}
Besides, we have by (\ref{eq:def_of_tau})
\begin{equation*}
IA_{10}(I)\sin(\varphi - I\tau_{\text{m}}^*) + A_{01}\sin(\pi-\tau_{\text{m}}^*) = 0,
\end{equation*}
which, introducing $\mu$ \eqref{eq:alpha}, is equivalent to
\begin{equation}\label{eq:def_tau_mu_neg}
\mu\alpha(I)\sin(\varphi - I\tau_{\text{m}}^*) + \sin(\pi-\tau_{\text{m}}^*) =0,
\end{equation}
or
\begin{equation}
-\mu\alpha(I)\sin(\varphi - I\tau_{\text{m}}^*) +\sin(-\tau_{\text{m}}^*) =0.\label{eq:equ_def_tau_mu_neg_and_pos}
\end{equation}

By \eqref{eq:rel_bet_s-tau_and_xi} and \eqref{eq:cristas_06} we have that $\pi - \tau_{\text{m}}^*  =  \xi_{\text{m}}(I,\varphi - I\tau_{\text{m}}^*)$ for $\pi/2 \leq\xi_{\text{m}}\leq 3\pi/2$ and therefore $-\pi/2\leq - \tau_{\text{m}}^*\leq \pi/2$.

By looking at (\ref{eq:def_tau_mu_neg}) and (\ref{eq:equ_def_tau_mu_neg_and_pos}), $\tau_{\text{m}}^*(I,\varphi,\pi)$ for $\mu$ is solution of the same equation as $\tau_{\text{M}}^*(I,\varphi,0)$ for $-\mu$, and lies in the same interval $-\pi/2 \leq -\tau_{\text{M}}^*\leq \pi/2$.  
Therefore $\tau_{\text{m}}^*(I,\varphi,\pi)$ for $\mu$ is equal to $\tau_{\text{M}}^*(I,\varphi,0)$ for $-\mu$.
From (\ref{eq:mel_pot_phi_mu_neg}), $L^*_{\mu,\m}(I,\varphi,\pi)$ satisfies 
\begin{eqnarray*}
L_{\mu,\text{m}}^*(I,\varphi,\pi) &=& A_{00} + A_{10}(I)\cos(\varphi - \tau_{\text{M}}^*(I,\varphi,0)) +(-A_{01})\cos(-\tau_{\text{M}}^*(I,\varphi,0))\\
&=& L_{-\mu,\text{M}}^*(I,\varphi,0).
\end{eqnarray*}

Since $L^*_{\mu,\m}(\cdot,\cdot,\pi)$ and $L^*_{-\mu,\M}(\cdot,\cdot,0)$ coincide, their derivatives too and this implies that $S_{\mu,\text{m}}(I,\varphi,\pi) =S_{-\mu,\text{M}}(I,\varphi,0) = \mathcal{S}_{-\mu,\text{M}}(I,\theta)$.
\end{itemize}
\end{proof}

The importance of the part b) of this lemma is that, concerning diffusion, the study for a positive $\mu$ using $\mathcal{S}_{\M}(I,\theta)$ is equivalent to the study for $-\mu$ using $S_{\m}(I,\varphi,\pi)$, i.e., if we ensure the diffusion for a positive $\mu$, we can ensure it for a negative one (just changing the scattering map).
Besides, since $\mathcal{S}_{\M}(I,\theta)$ symmetric in the variable $I$ (from the first part of the lemma), from now on we will consider always $I \geq 0$, $\mu >0$ and $\mathcal{S}_{\M}$.

Now we are going to describe the influence of the intersections between the crests and the \NH\, segments with respect to the parameter $\mu$ described in Proposition \ref{lem:crest} on the scattering map associated to such crests.

\subsubsection{Single scattering map: $\mu<0.625$\label{sub:3.2.4}}

As in \cite{Delshams2011}, assuming $\mu< 1/1.6 = 0.625$, the crests are horizontal and there is no tangency between $R_{\theta}(I)$ and $\C_{\text{M}}(I)$, so that $\tau^*_{\text{M}}(I,\theta)$ is well defined and by \eqref{eq:red_poi_fun_2} and \eqref{eq:our_meln_potential} the reduced Poincar\'{e} function takes the form
\begin{equation}
\mathcal{L}^*_{\text{M}}(I,\theta) = A_{00}+ A_{10}(I)\cos(\theta - I\tau^*_{\text{M}}(I,\theta)) + A_{01}\cos(-\tau^*_{\text{M}}(I,\theta)),\label{eq:red_poi_fun_theta}
\end{equation}
and therefore $\mathcal{S}_{\text{M}}(I,\theta)$ takes the form (\ref{eq:second_definition_SM}).

\paragraph{Example} To illustrate this construction, we fix $\mu = 0.6$.
In this case the crests are horizontal for all $I$, and we display $C_{\M}(I)$ parametrized by $\xi_{\M}$ (see (\ref{eq:cristas_06})) in Fig.\ref{fig:crest_and_SM_mu_06} for $I = 1.2$.
We can see how $R_{\theta}(I)$ intersects transversally $\C_{\text{M}}(I)$, as well as the phase space of scattering map $\mathcal{S}_{\M}$ generated by  this intersection given by the level curves of $\mathcal{L}_{\M}^*(I,\theta)$.

\begin{remark}\label{rem:itinial_points} Recall from Remark \ref{rem:s_parameter} that $s$ does not appear in the expression \eqref{eq:second_definition_SM_int} for $\mathcal{S}(I,\theta)$ and is a parameter in the expression \eqref{eq:first_definition_SM} for $S(I,\varphi,s)$.
Computationally, one difference is that in expression\eqref{eq:first_definition_SM}, once fixed a value of $s$, one throws from any ``initial point'' $(\varphi,s)$ the \NH\, segment $R(I,\varphi,s)$ until it touches the crest $\C(I)$ after a time $\tau^*(I,\varphi,s)$, obtaining a value for $L^*(I,\varphi,s)$ given by \eqref{eq:L^*}, while in expression \eqref{eq:second_definition_SM}, $s$ is fixed equal to $0$ or, equivalently, the initial point to throw the \NH\, segment $R_{\theta}(I)$ is of the form $(\theta,0)$ (see Fig. \ref{fig:variables}).
\end{remark}

\begin{figure}[h]
\centering
\subfigure[Intersection between $R_{\theta}(I)$ and $\C_{\text{M}}(I)$ (in blue) for $\mu =0.6$ and $I=1.2$. ]{\includegraphics[scale=0.27]{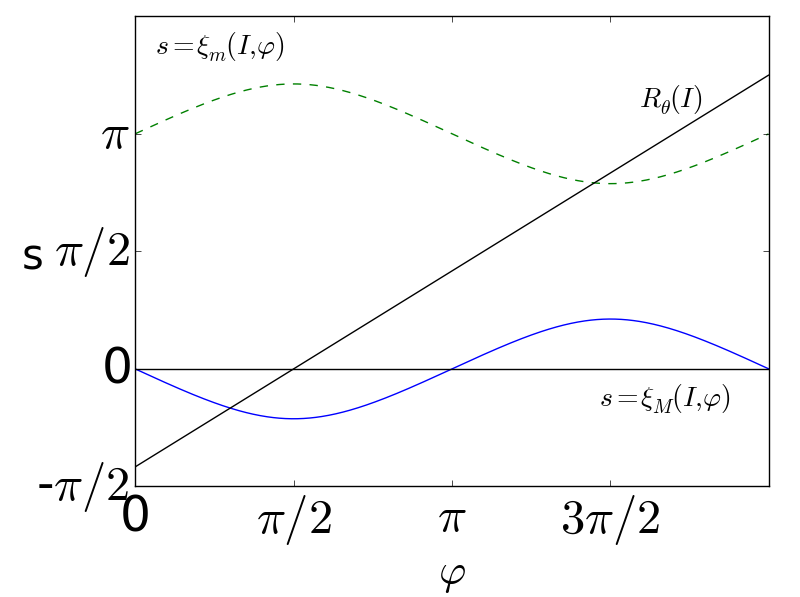}}
\qquad
\subfigure[The level curves of  $\mathcal{L}^*_{\M}(I,\theta)$.]{\includegraphics[scale=0.27]{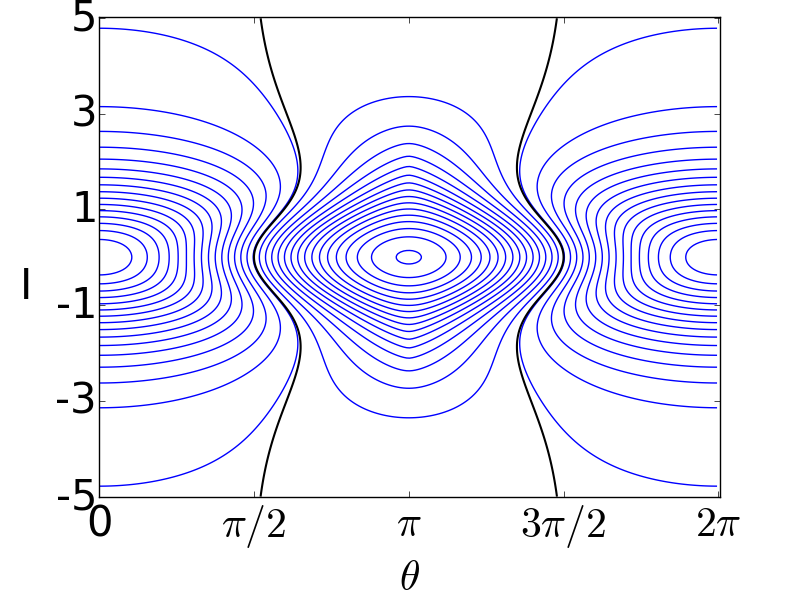}}
\caption{$R_{\theta}(I)\cap \C_{\M}(I)$ and $\mathcal{S}_{\M}(I,\theta)$.}
\label{fig:crest_and_SM_mu_06}
\end{figure}

\subsubsection{Multiple scattering maps: $0.625 \leq \mu \leq 0.97$ \label{sec:mult_scat_map}}

As said before, for $\mu <1/1.6 = 0.625$ and any value of $I$, the two crests  $C_{\M}(I)$ and $C_{\m}(I)$ are horizontal, and the \NH\, segment $R_{\theta}(I)$ intersects transversely each of them, giving rise to a unique scattering map $\mathcal{S}_{\M}$ and $\mathcal{S}_{\m}$ associated to each crest.
We will now explore larger values of $\mu$ to detect tangencies between $C(I)$ and $R_{\theta}(I)$, that is, when there exists $(\varphi,I)$ such that 
$$\frac{\partial \xi}{\partial \varphi}(I,\varphi)= 1/I,$$
where $\xi(I,\varphi)$ is the parameterization \eqref{eq:cristas_06} of the crest.

\paragraph{Tangencies between $\C(I)$ and $R_{\theta}(I)$ and multiple scattering maps\\\\}

We take $\C_{\text{M}}(I)$ parameterized by $\xi_{\M}$ as in \eqref{eq:cristas_06}.
For the other crest $\C_{\m}(I)$ is analogous.
Suppose that there exists a tangency point between $\C_{\text{M}}(I)$ and $R_{\theta}(I)$.
This is equivalent to the existence of $\psi$ such that $\partial\xi_\text{M}/\partial\psi(I,\psi) = 1/I.$
Using (\ref{eq:cristas_06}), this condition is equivalent to
\begin{equation}
-\frac{\mu\alpha(I)\cos\psi}{\sqrt{1-\mu^2\alpha(I)^2\sin^2\psi}} = \frac{1}{I},\label{eq:equa_tangency}
\end{equation}
where $\alpha(I)$ is introduced in \eqref{eq:alpha}.
Therefore 
$$\psi = \pm\arctan\left( \sqrt{\frac{I^2\mu^2\alpha(I)^2-1}{1-\mu^ 2\alpha(I)^2}}\right) + \pi,$$
where the expression under the square root is non-negative for $0.625\leq\mu\leq 0.97$ for some values of $I$ by Preposition \ref{lem:crest}. 
We are considering these values of $I$.

Equation (\ref{eq:equa_tangency}) implies $\cos\psi<0$, say $\psi\in\left(\pi/2,3\pi/2\right)$.
Denote the two tangent points by $\psi_1$ and $\psi_2$ and, without lost of generality, $\psi_1 \leq \psi_2$ with $\psi_1\in\left(\pi/2,\pi\right]$ and $\psi_2 = 2\pi - \psi_1 \in \left[\pi,3\pi/2\right)$. 

We consider the function relating the variables $\theta$ and $\psi$ (see Table \ref{tab:table_variables})
\begin{equation}
\theta (\psi) = \psi - I\xi_{\text{M}}(I,\psi),\label{eq:theta_psi}
\end{equation}
and define
\begin{equation*}
\theta_1 = \psi_1 - I\xi_{\text{M}}(I,\psi_1)\quad\quad \text{ and }\quad\quad \theta_2  =  \psi_2 - I \xi_{\text{M}}(I,\psi_2).
\end{equation*}

This function has only two critical points, $\psi_1$ and $\psi_2$.
Besides, we have 
$$I\frac{\partial \xi_{M}}{\partial \psi}(I,\psi) = -\frac{I\mu\alpha(I)\cos\psi}{\sqrt{1-\mu^2\alpha(I)^2\sin^2\psi}}< 0,\quad \forall \psi \in \left(0,\pi/2\right)\cup\left(3\pi/2,2\pi\right)$$
Therefore, $-I\partial \xi_M/\partial \psi > 0$, thus $d\theta/d\psi = 1 -I\partial \xi_M/\partial \psi >0 \quad \forall \psi \in \left(0,\pi/2\right)\cup\left(3\pi/2,2\pi\right).$

By continuity of $d\theta/d\psi$ and since $\theta(\psi)$ has only two critical points, we have 
\begin{eqnarray*}
\frac{d\theta}{d\psi} &>& 0 \quad  \forall \psi \in \left(0,\psi_1\right)\cup\left(\psi_2,2\pi\right)\\\label{eq:dtheta_dpsi}
\frac{d\theta}{d\psi} &<&0 \quad  \forall \psi \in \left(\psi_1,\psi_2\right).\nonumber
\end{eqnarray*}

Therefore $\theta_1 = \theta(\psi_1)\geq\theta_2 = \theta(\psi_2)$.
Note that $\theta(\left[\psi_2,2\pi\right])=\left[\theta_2,2\pi\right]$.
As $\theta_1 \in \left[\theta_2,2\pi\right]$, there is a $\tilde{\psi}_1 \in \left[\psi_2,2\pi\right]$ such that $\theta(\tilde{\psi}_1) = \theta_1$.
As $d\theta/d\psi$ is positive, $\tilde{\psi}_1$ is unique in that interval.
Analogously, we have $\tilde{\psi}_2 \in (0,\psi_1)$ such that $\theta(\tilde{\psi}_2)=\theta_2$.
We have $\tilde{\psi}_2\leq\psi_1\leq\psi_2\leq\tilde{\psi}_1$.
We can build, at least, three bijective functions:
\begin{eqnarray}
\theta_A :&  D_A:=\left[0,\tilde{\psi}_2\right]\cup\left(\psi_2,2\pi\right] &\longrightarrow \left[ 0 , 2\pi\right] \label{eq:three_bijections}\\
\theta_B: & D_B := \left[0,\psi_1\right)\cup\left[\tilde{\psi}_1,2\pi\right] &\longrightarrow \left[0,2\pi\right]\nonumber\\
\theta_C:& D_C:=\left[0,\tilde{\psi}_2\right]\cup\left(\psi_1,\psi_2\right)\cup\left[\tilde{\psi}_1,2\pi\right] &\longrightarrow \left[0,2\pi\right]\nonumber
\end{eqnarray}

If $\psi_1 < \psi_2$, that is, the tangency point is different from $\psi = \pi$, we have, at least, three scattering maps associated to $C_{\text{M}}$, the scattering map associated to $\mathcal{L}^*(I,\theta_j)$, $j = A,B,C$.

\begin{remark}
Those three scattering maps appear because the \NH\, line $R_{\theta}(I)$ intersects $C_{\text{M}}(I)$ three times for $\theta$ in the interval $(\theta_1,\theta_2)$.
\end{remark}

\begin{definition}
We call \emph{tangency locus} the set
\begin{equation*}
\left\{(I,\theta(\psi)):\frac{\partial \xi}{\partial \psi}(I,\psi)\,=\,\frac{1}{I} \text{ and }I\in\left[-I^*,I^*\right]\right\}.
\end{equation*}
\end{definition}

Fixed $I$ such that there exist tangencies, as we have seen before, there exist $\theta_1 \leq\theta_2$ such that $(I,\theta_1),(I,\theta_2)$ belong to the tangency locus.
We have that for any $\theta \notin \left(\theta_1,\theta_2\right)$ there is only one scattering map. 
But we have three different scattering maps for $\theta \in \left(\theta_1,\theta_2\right)$.
We can see this behaviour on the example below.

\paragraph{Example}
We illustrate the scattering maps of $C_{\text{M}}(I)$ for $\mu = 0.9$ in Fig. \ref{fig:multi_scat_maps}.
We can see the three scattering maps and we emphasize their difference showing a zoom around the tangency locus.
In this zoom, we can see curves with three different colors.
Each color represents a different scattering map.

\begin{figure}[h]
\centering
\subfigure[The three intersections between $R_{\theta}(I)$ and $C_{\text{M}}(I)$ for $\mu =0.9$ , $I=1.5$.]{\includegraphics[scale=0.27]{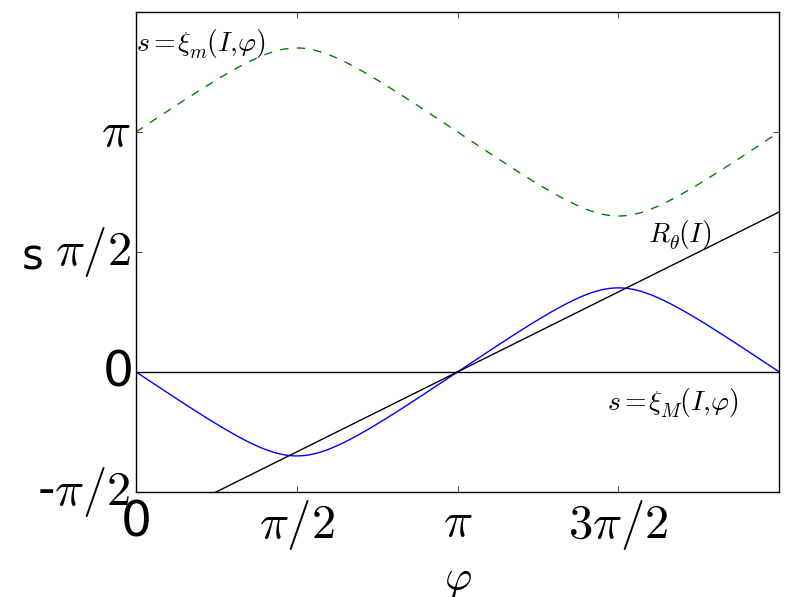}}
\qquad
\subfigure[Level curves of $\mathcal{L}_{\M}^*(I,\theta)$.]{\includegraphics[scale=0.27]{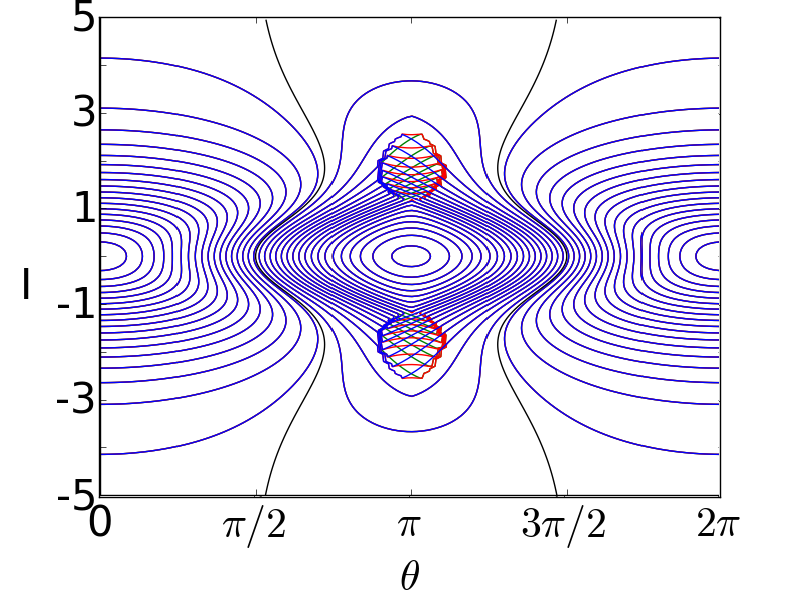}}
\subfigure[Zoom around the tangency locus]{\includegraphics[scale=0.27]{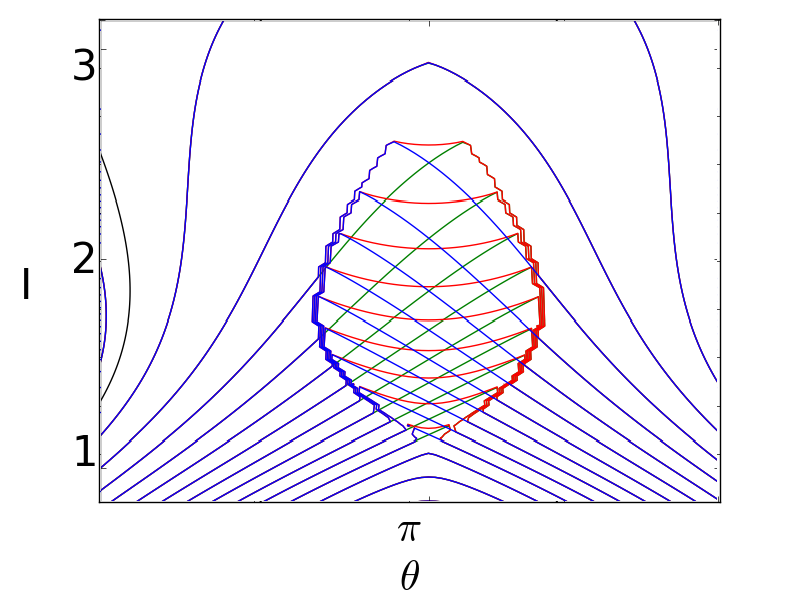}}
\caption{Tangencies: Multiple scattering maps.}
\label{fig:multi_scat_maps}
\end{figure}

\subsubsection{The scattering maps ``with holes'': $\mu> 0.97$}\label{sec:holes}

We study now the case when $\mu$ is large enough such that $\mu\alpha(I)>1$ for some $I$, that is, for $\mu > 0.97$.
In this case, the horizontal crests become vertical crests for some values of $I$.
But locally, the structure of the parameterizations  $\xi_{\text{M}}$ and $\xi_{\text{m}}$ are preserved, that is, even if the crests are vertical from a global view-point, these crests are formed by pieces of horizontal crests. 
So, some intersections between $R_{\theta}(I)$ and $C(I)$ parameterized by the vertical parameterization $\eta$, given in (\ref{eq:eta_definition}), can be seen, indeed, as intersections between $R_{\theta}(I)$ and $C(I)$ parameterized by $\xi$, given in (\ref{eq:cristas_06}).
Using this idea, we can be extend the scattering map associated to the reduced Poincar\'{e} function, given in \eqref{eq:red_poi_fun_psi}, for the values of $(I,\varphi)$ such that $\mu\alpha(I)>1$ but $\left|\mu\alpha(I)\sin\varphi\right|<1$.
For some values of $\varphi$ like $\varphi = \pi/2,3\pi/2$, this is not possible, and for those values of $\varphi$ ``holes'' appear in the definition of the scattering map when the horizontal parameterization $\xi$ is used.
\begin{remark}
For the diffusion, a priori, the existence of such values can be a problem.
One can avoid these holes using the inner map, or using another scattering map associated to the vertical parameterization $\eta$ given in \eqref{eq:eta_definition}.
\end{remark}

\paragraph{Example}
We illustrate this case in Fig. \ref{fig:scat_with_holes}.
We display in (a) an example of intersection between $R_{\theta}(I)$ and $C_{\M}(I)$ and in (b) the level curves of $\mathcal{L}_{\M}^*(I,\theta)$ (recall that they provide an approximation to the orbits of the scattering map $\mathcal{S}_{\M}(I,\theta) )$. 
The green region in (b) is the region where the scattering map is not defined, that is, for a point $(I,\theta)$ in this region, $R_{\theta}(I)$ does not intersect $C_{\M}(I)$.
\begin{figure}[h]
\centering
\subfigure[Intersection between $C_{\M}(I)$ and $R_{\theta}(I)$. $\mu = 1.5$ and $I = 1$.]{\includegraphics[scale=0.27]{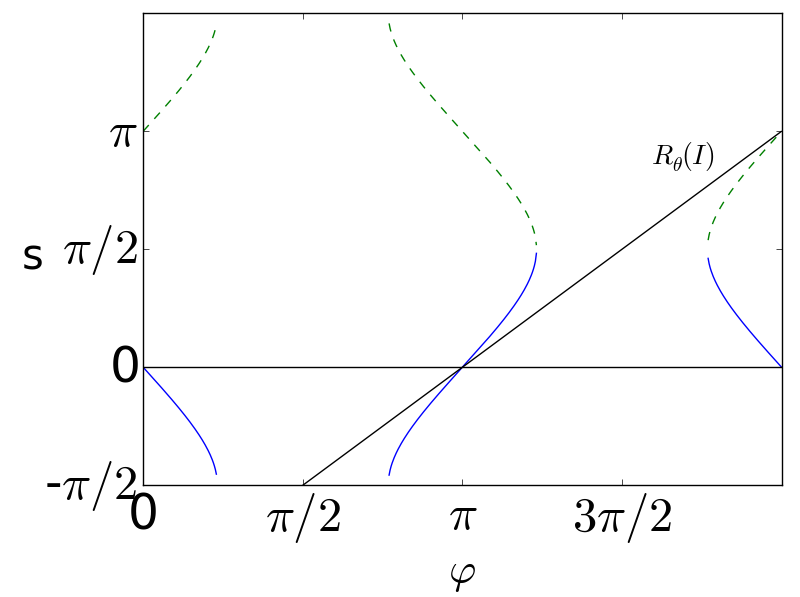}}
\qquad
\subfigure[Level curves of $\mathcal{L}_{\M}^*(I,\theta)$. In green, the region which the level curves are not defined.]{\includegraphics[scale=0.27]{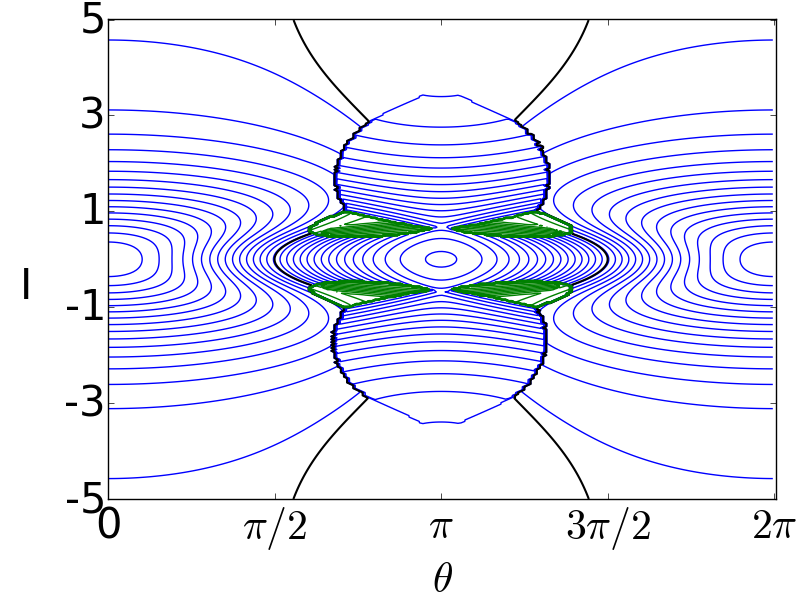}}
\caption{Scattering map with holes}\label{fig:scat_with_holes}
\end{figure}
\subsubsection{Summary of the scattering maps}
Taking into account the results of the last three sub-subsections \ref{sub:3.2.4}--\ref{sec:holes} on the primary scattering map $\mathcal{S}_{\M}$, $\mathcal{S}_{\m}$ for $\mu>0$ as well as Lemma \ref{lem:geometrical_lemmas}  we can complete Proposition \ref{lem:crest}. 
\begin{theorem}\label{new_theorem}
Consider the crests $C(I)$ defined in \eqref{eq:def_crista} and the \NH\, lines $R(I,\theta)$ defined in \eqref{eq:line_segments_I_theta}
\begin{itemize}
\item For  $0<\left|\mu\right|< 0.625$ the two crests are horizontal and the intersection between any crest and any \NH\, lines is transversal.
There exist two primary scattering maps $\mathcal{S}(I,\theta)$ defined on the whole range of $\theta\in\mathbb{T}$.
\item For $0.625 \leq \left|\mu\right| \leq 0.97$ the two crests are still horizontal, but for some values of $I$ there exist two \NH\, lines $R_{\theta_1}(I)$, $R_{\theta_2}(I)$ which are geometrically tangent to the crests.
There exist two or six scattering maps defined for $\theta\neq\theta_1,\,\theta_2$.
\item For $\left|\mu\right| >0.97$, the same properties stated in b) hold, except that for some bounded interval of $\left|I\right|$ there exists a sub-interval of $\theta\in\mathbb{T}$ such that the scattering maps are not defined.
\end{itemize}
\end{theorem}

\section{Arnold diffusion}\label{sec:Arnold_diffusion}

From now on, our goal will be the study of Arnold diffusion using adequately chosen scattering maps.
For this diffusion, it will be important to describe the level curves of the reduced Poncar\'{e} function $\mathcal{L}^*(I,\theta)$, since the scattering map is to up an error $\mathcal{O}(\varepsilon^2)$, the $-\varepsilon$ time flow of the Hamiltonian $\mathcal{L}^*(I,\theta)$.
Among the level curves of $\mathcal{L}^*(I,\theta)$, we will first describe two candidates to fast diffusion, namely the ones of equation $\mathcal{L}^*(I,\theta) = A_{00} + A_{01}$, that will be called ``highways''.
Indeed, such highways will be taken into account in the two theorems about the existence of diffusion that will be proven in this section.

In the next proposition we prove that $\mathcal{L}^*(I,\theta) = A_{00} + A_{01}$ is a union of two ``vertical'' curves in a rectangle $\mathbb{T}\times B$, that is, it can be written as $\text{H}_\text{l}\cup \text{H}_\text{r}$ where $\text{H}_k =  \left\{(I,\theta_k(I)):I\in B\right\}$, $\theta_k(I)$ is a smooth function, and the index $k$ takes the value l for left ($0<\theta_\text{l}(I)<\pi$) o r for right ($\pi<\theta_\text{r}(I)<2\pi$). 
To prove this, we only need to prove that 
$$\frac{\partial \mathcal{L}_{\M}^*}{\partial \theta}(I,\theta) \neq 0 \quad \forall (I,\theta) \in\left\{(I,\theta):\mathcal{L}_{\M}^*(I,\theta) = A_{00}+A_{01} \text{ and }  I\in B\right\}.$$

\subsection{A geometrical proposition: The level curves of $\mathcal{L}^*(I,\theta)$}

\begin{proposition}\label{pro:geometrical_proposition}
Assuming $a_{10}\,a_{01}\,\neq\,0,$ the level curve $\mathcal{L}^*(I,\theta) = A_{00}+A_{01}$ of the reduced Poincar\'{e} function \eqref{eq:red_poi_func_1} is a union of two ``vertical'' curves on a cylinder $(\theta,I)\in\mathbb{T}\times B$, where the set $B$ is given by
\begin{itemize}
\item for $\left|\mu\right|<0.625$, $B$ is the real line.
\item for $0.625\leq\left|\mu\right|$, $B = \left(-\infty,-I_{++}\right)\cup\left(-I_+,I_+\right)\cup\left(I_{++},+\infty\right),$ where
$$I_{++} = \max\left\{I> 0:\frac{I^3\sinh(\pi/2)}{\sinh(I\pi/2)} = \frac{1}{\left|\mu\right|}\right\}$$
and 
\begin{itemize}
\item $I_{+} = \min\left\{I > 0:I^3\sinh(\pi/2)/\sinh(I\pi/2) = 1/\left|\mu\right|\right\},\,\text{for } \left|\mu\right|\leq 1$
\item $I_{+} = \min\left\{I > 0:I^2\sinh(\pi/2)/\sinh(I\pi/2) = 1/\left|\mu\right|\right\},\, \text{for }\left|\mu\right|\geq 1$
\end{itemize}
\end{itemize}
\end{proposition}

\begin{proof}
Consider the real set $A$:
\begin{equation}
A = \left\{I\geq 0 : \left|\alpha(I)\right|\leq \frac{1}{\left|\mu\right|}\right\}.\label{eq:def_of_the_set_A}
\end{equation}
For $I\in A$, the maximum crest $\C_{\M}(I)$ is horizontal and can be parameterized by the expression (\ref{eq:cristas_06}) and $\xi_{\text{M}}(I,0)=\xi_{\M}(I,\pi)=\xi_{\M}(I,2\pi) = 0$. 

Consider now the subset of $A$
\begin{equation}
B = \{I \in A : \text{there is no tangency between $\text{C}_{\M}(I)$ and $R_{\theta}(I)$}\}.\label{eq:def_of_the_set_B}
\end{equation}
As already mentioned, for $I\in B$ one has $\partial \xi_{\M}/\partial \psi(I,\psi)\neq 1/I,\quad\forall \psi\,\in\,\left[0,2\pi\right]$.
In particular, for $I\in B$ the change \eqref{eq:def_theta_psi} $\psi\in\mathbb{T}\mapsto \theta = \psi - I\xi(I,\psi)\in\mathbb{T}$ is smooth with inverse
\begin{equation}\label{def:psi}
\psi = \theta - I\tau_{\M}^*(I,\theta)\qquad \forall \theta\in\mathbb{T}.
\end{equation}
Then we can rewrite for $I\in B$ and $\theta\in\mathbb{T}$ the reduced Poincar\'{e} function $\mathcal{L}_{\M}^*(I,\theta)$ of \eqref{eq:red_poi_fun_theta} in terms of this variable $\psi$ as 
$$\mathfrak{L}_{\M}^*(I,\psi) = A_{00} + A_{10}(I)\cos\psi + A_{01}\cos\xi_{\M}(I,\psi).$$  

Notice that $\mathfrak{L}_{\M}^*(I,\psi)$ is well defined for all $(I,\psi)\in A\times\mathbb{T}$ and it is immediate to see that for any $I\in A$ there exists exactly one $\psi_0\in(0,\pi)$ and another one $\psi_1\in(\pi,2\pi)$ such that $\mathfrak{L}_{\M}^*(I,\psi_0) = \mathfrak{L}_{\M}^*(I,\psi_1)= A_{00}+A_{01}$.
Restricting now to $I\in B$, the same property holds for $\mathcal{L}^*(I,\theta)$, since the relation between $\theta$ and $\psi$ is a change of variables sending $\theta = 0,\,\pi$ to $\psi = 0,\,\pi$ respectively.
In other words, introducing the projection $\Pi:\mathbb{R}\times\mathbb{T}\rightarrow\mathbb{R}$, $\Pi(I,\theta) = I$, $B\subset\Pi\left(\mathcal{L}^{*-1}_{\M}(A_{00}+A_{01})\right)$.

We can characterize $B$ defined in \eqref{eq:def_of_the_set_B} by the following property
\begin{equation}\label{eq:def_of_beta}
I\in B \Leftrightarrow \beta(I):=I\alpha(I)<\frac{1}{\left|\mu\right|}.
\end{equation}
Indeed, by definition \eqref{eq:def_of_the_set_A}, $A$ is characterized by $I\in A \Leftrightarrow \alpha(I)\leq 1/\left|\mu\right|$, where $\alpha(I)\geq 0 $ is defined in \eqref{eq:alpha}, and it satisfies $\lim_{I\rightarrow 0^+}\alpha(I) = 0 = \lim_{I\rightarrow +\infty}\alpha(I)$ and it has a unique positive critical point $I_{\alpha} \approx 1.219$ which is a global maximum, see Fig.\ref{fig:alpha_beta}.
Therefore
\begin{equation} 
\alpha(I)\leq\alpha(I_{\alpha}) = \frac{1}{0.97}\sim 1.03.\label{eq:upper_bound_alpha}
\end{equation}
On the other hand, for $I\in A$ there exist tangencies between $\C_{\M}(I)$ and $R_{\theta}(I)$ as long as the condition \eqref{eq:equa_tangency} holds, which can only take place for $\left|I\alpha(I)\right|\geq1/\mu$, which justifies the characterization \eqref{eq:def_of_beta} for $B$.

The function $\beta(I)$ is very similar to  $\alpha(I)$, that is, $\beta(I)$ is always positive for $I>0$ , it has a unique positive critical point $I_{\beta} = 1.9$ and $\beta(I)\rightarrow 0$ as $I\rightarrow 0$ and $I\rightarrow +\infty$.
This positive critical point is a global maximum point,
\begin{equation} 
\beta(I) \leq\beta(I_{\beta})=1.6.\label{eq:upper_bound_beta}
\end{equation}
Besides, by (\ref{eq:def_of_beta}), for $I<1$, $\beta(I)<\alpha(I)$, $\beta(1)=\alpha(1) = 1$ and for $I>1$, $\beta(I)>\alpha(I)$.
See Fig. \ref{fig:alpha_beta}.

\begin{figure}
\centering
\includegraphics[scale=0.2]{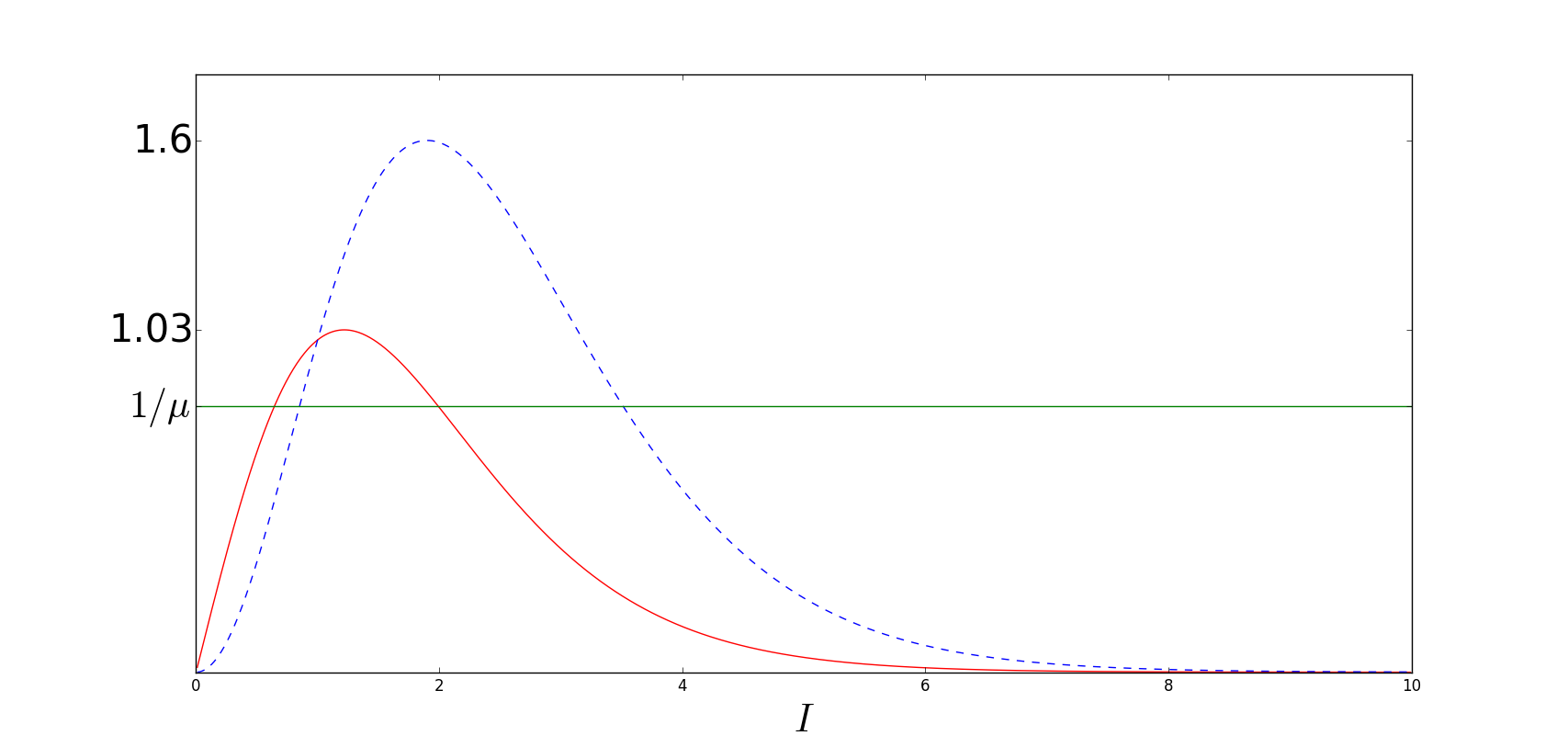}
\caption{Graph of $\alpha(I)$, in red, and $\beta(I)$, in blue (dashed).}\label{fig:alpha_beta}
\end{figure}

Now we consider the three case of the proposition, that is, \textbf{1)} $\left|\mu\right| < 0.625$, \textbf{2)} $ 0.625\leq \left|\mu\right| \leq 1 $ and \textbf{3)} $\left|\mu\right| \geq 1$.

\begin{itemize}
\item{\textbf{Case 1}}
$\left|\mu\right| < 0.625$, that is, $1/\left|\mu\right|  > 1.6$.
Then, by (\ref{eq:upper_bound_alpha}) and (\ref{eq:upper_bound_beta}),
$$\alpha(I)\leq 1.03 <\frac{1}{\mu} \quad\text{ and }\quad\beta(I)\leq 1.6<\frac{1}{\left|\mu\right| },$$
for all $I$, that is, for $I>0$, $B = \left[0,+\infty\right).$

\item{\textbf{Case 2}}  Note that for $ 0.625\leq \left|\mu\right| <0.97$ 
$$\alpha(I)\leq 1.03 = 1/0.97<\frac{1}{\left|\mu\right|}\leq 1.6 = \beta(I_{\beta}),$$
and by (\ref{eq:upper_bound_alpha}), $A = \left[0,+\infty\right)$.
But, now $\beta(I_{b})\geq 1/\left|\mu\right|$.
Then there exist two values $I \in A$ such that $\beta(I) =1/\left|\mu\right|$.
Define
\begin{equation}
I_{+} = \min\left\{I:\beta(I) = 1/\left|\mu\right|\right\}\quad\quad\text{and}\quad\quad I_{++} = \max\left\{I:\beta(I) = 1/\left|\mu\right|\right\}.\label{I_+_1_and_I_++_1}
\end{equation}
By the characterization \eqref{eq:def_of_beta} of the set $B$ we have $B=\left[0,I_{+}\right)\cup\left(I_{++},+\infty\right)$.

For $0.97 \leq \left|\mu\right|\leq 1$, there exist $I_{a}<I_{\bar{a}}$ such that $\alpha(I_{j}) = 1/\left|\mu\right|$, $j\in \{a,\bar{a}\}$ and $A = \left[0,I_a\right)\cup\left(I_{\bar{a}},+\infty\right)$. Analogously, there exist $I_{b} < I_{\bar{b}}$ such that  $\beta(I_{j}) = 1/\left|\mu\right|$, $j\in \{b,\bar{b}\}$.
As $I_{b}\leq I_{a}$ and $I_{\bar{a}}<I_{\bar{b}}$, we have $B = \left[0,I_{b}\right)\cup\left(I_{\bar{b}},+\infty\right)$, see Fig. \ref{fig:alpha_beta}. 
But this the equivalent to $B=\left[0,I_{+}\right)\cup\left(I_{++},+\infty\right)$, where $I_{+}$ and $I_{++}$ are given by \eqref{I_+_1_and_I_++_1}.

\item{\textbf{Case 3}} This case is similar to the \textbf{Case 2} for $0.97 \leq \left|\mu\right|\leq 1$.
But now, as $\left|\mu\right| \geq 1$, we have $I_{a}\leq I_{b}$.
So, in this case we have $B =\left[0,I_{a}\right)\cup\left(I_{\bar{b}},\infty\right)$, or $B = \left[0,I_+\right)\cup\left(I_{++},+\infty\right),$
where $I_{+} = \min\left\{I:\alpha(I)= 1/\left|\mu\right|\right\}$ and $I_{++} = \max\left\{I:\beta = 1/\left|\mu\right|\right\}$.
\end{itemize}

Finally, we see that $\mathcal{L}^*_{\M}(I,\theta) = A_{00} + A_{01}$ is composed by two curves in rectangles $(\theta,I)\in \left((0,\pi)\cup(\pi,2\pi)\right)\times B$.
This is equivalent to prove that the derivative of this curve with respect to the variable $\theta$ is different from $0$ for all $I$ in $B$.
For any $I\in B$, we compute the expression for $\partial \mathcal{L}^*_{\M}/\partial \theta(I,\theta)$ which using \eqref{eq:def_of_tau} and the change of variables \eqref{def:psi} takes the form
\begin{equation}
\frac{\partial \mathcal{L}^*_{\M}}{\partial \theta}(I,\theta) = -A_{10}(I)\sin(\psi),\label{eq:der_red_poi_fun_theta}
\end{equation}
and never vanishes for $\psi\in(0,\pi)\cup(\pi,2\pi)$, or equivalently, for $\theta\in(0,\pi)\cup(\pi,2\pi)$.
Then $\mathcal{L}^*_{\M}(I,\theta) = A_{00} + A_{01}$ is composed by two vertical curves on $B$.

As we have seen in  Lemma \ref{lem:geometrical_lemmas}, $\mathcal{L}^*(-I,\theta) = \mathcal{L}^*(I,\theta)$.
Then, the level curve $\mathcal{L}^*_{\M}(I,\theta) = A_{00}+A_{01}$  is also defined for $I<0$, which concludes the proof.
\end{proof}

\begin{remark}
Using the expressions above for $I_+$ and $I_{++}$ one can check that
$$I_+ \sim \frac{\pi}{2\left|\mu\right|\sinh(\pi/2)} \quad\text{ and }\quad I_{++} \sim \left(\frac{2}{\pi}\right)\log(\left|2\sinh(\pi/2)\mu\right|),\text{ as } \left|\mu\right|\rightarrow +\infty.$$ 
\end{remark}

\begin{definition}\label{def:highways}
We call \emph{highways} the two curves  $H_{\text{l}}\subset (0,\pi)\times \mathbb{T}$  and $H_{\text{r}}\subset(\pi,2\pi)\times \mathbb{T}$ such that $\mathcal{L}^*(I,\theta) = A_{00}+A_{01}$.
By Proposition \ref{pro:geometrical_proposition}, they exist at least for $I\in\left(-\infty,-I_{++}\right) \cup\left(-I_+,I_+\right)\cup\left(I_{++},+\infty\right)$ for $\left|\mu\right|\geq 0.615$ and for any value $I$ for $\left|\mu\right|<0.625$.
If $a_{10}>0$, by \eqref{eq:der_red_poi_fun_theta}, $\partial\mathcal{L}^*/\partial \theta$ is positive (respectively negative) along the highway $H_{\text{r}}$ (resp. $H_{\text{l}}$).
If $a_{10}<0$, change $H_{\text{l}}$ to $H_{\text{r}}$. 
\end{definition}

\begin{figure}[h]
\centering
\includegraphics[scale=0.27]{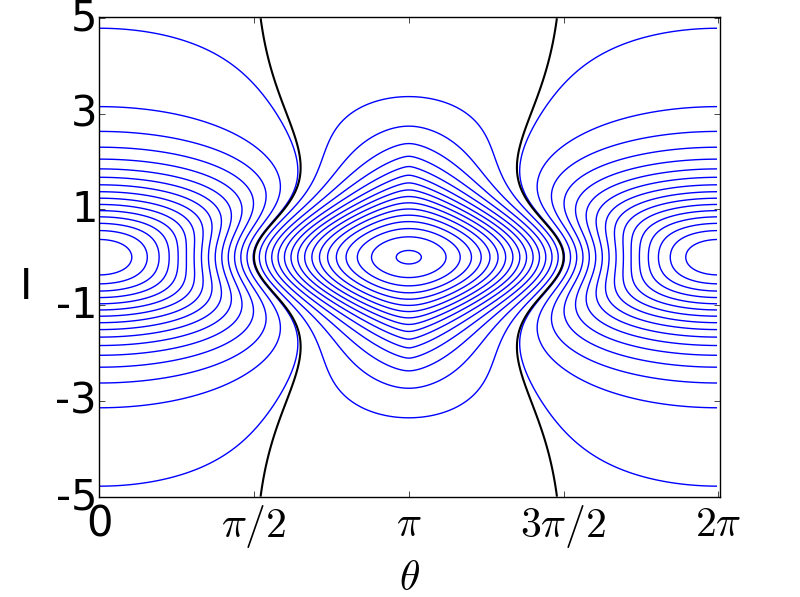}
\caption{\emph{Highways} in black for $\mu = 0.6$.}\label{fig:highways_example}
\end{figure}
\subsection{Results about global instability}

Now we are going to prove two results about existence of the diffusion phenomenon in our model.
The first one is a direct application of the geometrical Proposition \ref{pro:geometrical_proposition} just proved and describes the diffusion that takes place close to the highways.
The second is a more general type of diffusion, valid also for the values of the action $I$ where there are no highways.

\subsubsection{Diffusion close to highways}

\begin{theorem}\label{teo:easy_diffusion}
Assume that $a_{10}\,a_{01}\,\neq\,0$ in the Hamiltonian \eqref{eq:hamiltonian_int}$+$\eqref{eq:perturbation_int}.
Then, for any $I^*$ there exists $\varepsilon^{*}=\varepsilon^{*}(I^{*})>0$ such that for
$0 < \varepsilon < \varepsilon^{*}$, there exists a trajectory
$(p(t),q(t),I(t),\varphi(t))$ such that for some $T>0$
\[ I(0) \leq -I^*; \qquad I(T) \geq I^{*},\]
where the admissible values for $I^* = I^*(\mu)$ satisfy
\begin{itemize}
\item For $\left|\mu\right|<0.625$, $I^*$ is arbitrary $I^* \in\left(0,+\infty\right)$.
\item For $ 0.625\leq\left|\mu\right|\leq 1$, $I^*\in\left(0,I_{+}\right)$, where
$I_{+} = \min\{I>0:I^3\sinh(\pi/2)/\sinh(\pi I/2) = 1/\left|\mu\right|\}$.
\item For $\left|\mu\right|\geq 1$, $I^*\in\left(0,I_+\right)$, where $I_+=\{I>0:I^2\sinh(\pi/2)/\sinh(\pi I/2) = 1/\left|\mu\right|\}.$
\end{itemize}
\end{theorem}

\begin{proof}
Recall that the reduced Poincar\'e  function, given in (\ref{eq:red_poi_fun_theta}), is
$$\mathcal{L}^*_{\M}(I,\theta) = A_{00} + A_{10}(I)\cos(\theta - I\tau^*_{\M}(I,\theta)) +  A_{01}\cos(-\tau^*_{\M}(I,\theta)).$$
During this proof, we denote $\tau^*_{\M}(I,\theta)$ simply by $\tau_{\M}^*$.
For $\varepsilon $ small enough, the scattering map $\mathcal{S}_{\M}(I,\theta)$ takes the form \eqref{eq:second_definition_SM_int} for $\mathcal{L}^* = \mathcal{L}^*_{\M}$, so that orbits under the scattering map are contained in the level curves of the reduced Poincar\'{e} function $\mathcal{L}_{\M}^*$ up to error of $\mathcal{O}(\varepsilon^2)$.

Proposition \ref{pro:geometrical_proposition} ensures the existence of the highways as two vertical level curves $\mathcal{L}^*_{\M}(I,\theta) = A_{00} + A_{01}$ for $I$ in 
\begin{itemize}
\item $\left(-\infty, +\infty \right)$ for $\left|\mu\right|<0.625$.
\item $\left(-I_{+},I_{+}\right),$ where
\begin{itemize}
\item $I_+ = \min\{I>0:I^3\sinh(\pi/2)/\sinh(\pi I/2) = 1/\left|\mu\right|\}$ for  $0.625\leq\left|\mu\right|\leq 1$;
\item $I_+ = \min\{I>0:I^2\sinh(\pi/2)/\sinh(\pi I/2) = 1/\left|\mu\right|\}$ for $\left|\mu\right|\geq 1$.
\end{itemize}
\end{itemize}

Take $a_{10}>0$.
Then given $I^* >0$ (with the restriction $I^*<I_+$ if $\left|\mu\right| >0.625$), $\partial \mathcal{L}^*_{\M} >0$ along the highway $H_{\text{r}}$.
Note that $(I_{0},\theta_{0}) := (0,3\pi/2)\in H_{\text{r}}$.
Taking any $(I_i,\theta_i)\in H_{\text{r}}$, $I_i>0$,  its image under the scattering map $(\widetilde{I}_{i+1},\widetilde{\theta}_{i+1}) = \mathcal{S}_{\M}(I_i,\theta_i)$ satisfies $\widetilde{I}_{i+1} - I_{i} = \mathcal{O}(\varepsilon)>0$ and is $\mathcal{O}(\varepsilon^2)$-close to $H_{\text{r}}$.
Using the inner map on $\tilde{\Lambda}$, we find $(I_{i+1},\theta_{i+1})= \phi_{t_{i+1}}(\widetilde{I}_{i+1},\widetilde{\theta}_{i+1})\in H_{\text{r}}$ with $I_{i+1} - I_{i} = \mathcal{O}(\varepsilon)>0$.
Continuing recursively in this way, we get a pseudo-orbit $\{(I_{i},\theta_i), i = 0,\dots,N\}\subset H_{\text{r}}$ with $I_{N}\geq I^*$ formed by applying successively the scattering map and the inner map.
Using the symmetry of $H_{\text{r}}$, introducing $I_{i} = -I_{i}$ for $i<0$, we have the pseudo-orbit $\{(I_i,\theta_i),\left|i\right| \leq N\}\subset H_{\text{r}}$.
Using standard shadowing results in \cite{fontich2000,fontich2003} (changing slightly $I_{i}$ to obtain an irrational frequency of the inner map, if necessary ) or newer results like the corollary 3.5 of \cite{Gidea2014}, there exists a trajectory of the system such that for some $T$, $I_{0}\leq -I^*$ and $I(T)\geq I^*$.
If $a_{10}<0$, changing $H_{\text{r}}$ to $H_{\text{l}}$ all the previous reasoning applies.
\end{proof}

\begin{figure}[h]
\centering
\includegraphics[scale=0.27]{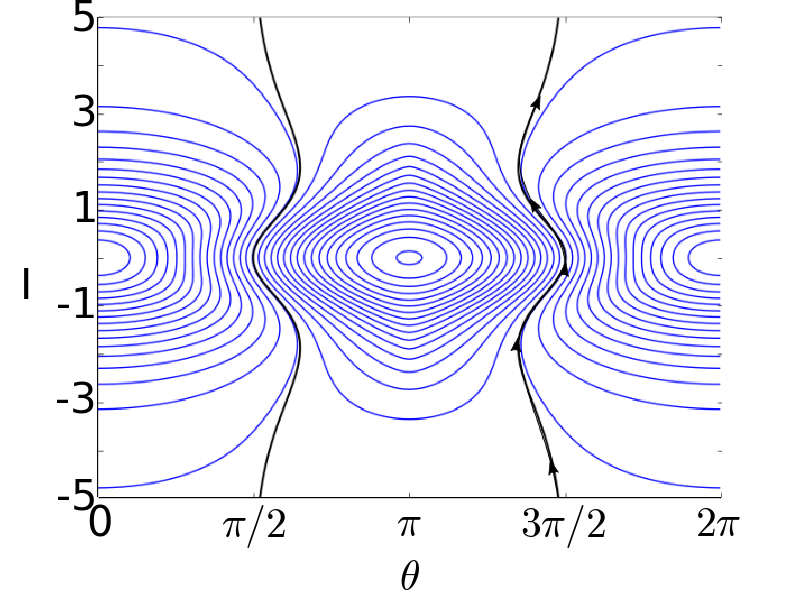}
\caption{The diffusion trajectory in $\mathcal{S}_{\M}$ for $\mu = 0.6$.}
\end{figure}

\subsubsection{The general diffusion\label{gen_diff}}

Now we present a theorem that ensures the diffusion for all values of the parameter $a_{10},a_{01}$ (as long as $a_{10}a_{01}\neq 0$) and for any value of $I^*$.
Beside we prove it using the geometrical properties of the scattering map that we have explored up to now.

\begin{theorem}\label{teo:main_theorem}
Assume that $a_{10}\,a_{01}\,\neq\,0$ in the Hamiltonian \eqref{eq:hamiltonian_int}$+$\eqref{eq:perturbation_int}.
Then, for any $I^{*}>0$, there exists $\varepsilon^{*}=\varepsilon^{*}(I^{*})>0$  such that for any $\varepsilon$,
$0 < \varepsilon < \varepsilon^{*}$, there exists a trajectory
$(p(t),q(t),I(t),\varphi(t))$ such
that for some $T>0$
$$I(0)\leq -I^*<I^*\leq I(T).$$
\end{theorem}

\begin{proof}
Our proof consists on showing the existence of adequate orbits under several scattering maps, whose orbits will be given approximately by the level curves of the corresponding reduced Poincar\'{e} functions, in such a way the value of $I$ will be increasing.
Later on, we will combine them with orbits under the inner map to produce adequate pseudo-orbits for shadowing.

We begin with the simplest case.
Assume $\left|\mu\right| <0.625$.
In this case the highways, by Proposition \ref{pro:geometrical_proposition}, are defined for any value of $I\in\mathbb{R}$ and  Theorem \ref{teo:easy_diffusion} ensures the diffusion phenomenon.

We now assume $0.625\leq \left|\mu\right|\leq 0.97$.
In this case for some value of $I$ there may exist tangencies between the crests $\C_{\M}(I)$ and the \NH\, lines $R_{\theta}(I)$.
Again by Proposition \ref{pro:geometrical_proposition}, in this  case the highways are defined for all $I\in\left(-\infty,-I_{++}\right)\cup\left(-I_+,I_+\right)\cup\left(I_{++},+\infty\right)$ where $0<I_+\leq I_{++}$.
The case $I^*\in\left(0,I_{+}\right)$ is contained in the result of Theorem \ref{teo:easy_diffusion}.
So, we are going to consider $I^*\in\left[I_{+},+\infty\right)$. 

As before, we have one $\mathcal{S}_{\M}$-orbit contained in one highway where $I$ is increasing.
We have to study the region of $I$ where the highways are not defined.

Our strategy is proving the existence of a scattering map in the side of $\theta$ where the $I$ is increasing, that is, for $\theta \in \left(0,\pi\right)$ or $\theta \in \left(\pi,2\pi\right)$ (this depends of $\text{sign}(a_{10})$) where $\partial \mathcal{L}^*_{\M}/\partial \theta$ is positive.
Then, we will use the inner map (or another scattering map $S'$) for changing of pseudo-orbit (level curve) of $\mathcal{L}^*_{\M}$.
In this way, we continue the growth of $I$. 

For any $I\in\left(-I_{++},-I_{+}\right)\cup\left(I_{+},I_{++}\right)$, there exist tangencies between $C_{\M}(I)$ and $R_{\theta}(I)$, i.e., there exists $\psi$ such that $\partial \xi_M/\partial \psi = 1/I$, and therefore there exist three different scattering maps.

Consider the case with $\mu >0$.
As we have seen in Subsection \ref{sec:mult_scat_map}, $\psi\in\mathbb{T}\mapsto\theta\in\mathbb{T}$ given in (\ref{eq:theta_psi}) is no longer a change of variables, but we have three bijections $\theta_i: D_i(I) \rightarrow \mathbb{T}$, $i\in\{\text{A},\text{B},\text{C}\}$ (see \eqref{eq:three_bijections}).
And for each bijection we have a scattering map associated to it.
Among these three scattering maps, we will chose only one for the diffusion.
Consider first the case $a_{10}>0$ (recall that the highway $H_{\text{r}}$ goes from $-I_+$ toward $I_+$).
We chose for instance, the scattering map associated to the reduced Poincar\'{e} function $\mathcal{L}^*_{\M,\text{A}}(I,\theta) = \mathcal{L}_{M}^*(I,\theta_{\text{A}}(\psi))$, $\psi \in D_{\text{A}}(I)$ since
$$\frac{\partial \mathcal{L}^*_{\M} }{\partial \theta}(I,\theta_{A}(\psi)) = -A_{10}(I)\sin(\psi) >0\quad  \text{for } \psi \in D_{\text{A}}(I)\cap\left(\pi,2\pi\right)$$
and therefore the iterates under the scattering map $\mathcal{S}_{\M.A}(I,\theta)$ \eqref{eq:second_definition_SM_int} associated to $\mathcal{L}^*_{\M,\text{A}}(I,\theta)$ increase the values of $I$ for $\theta\in(\pi,2\pi)$.
Notice that by definition of $D_{A}(I)$ for $\psi\in D_{\text{A}}(I)\cap \left(\pi,2\pi\right) = \left(\psi_2,2\pi\right)$ with $\psi_2 \in \left(\pi,3\pi/2\right)$ (see Subsection \ref{sec:mult_scat_map}) there are no tangencies between the crest and the \NH\, segment.

We can now proceed in the following way.
We first construct a pseudo-orbit $\{(I_i,\theta_i): i = 0,\dots,N_1\}\subset H_{\text{r}}$ with $I_0 = 0$  and $I_{\text{N}} = I_+$, as in the proof of Theorem \ref{teo:easy_diffusion}.
Note that all these points lie in the same level curve of $\mathcal{L}_{\M}^*$, that is, $\mathcal{L}_{\M}^*(I_i,\theta_i) = A_{00} + A_{01}$, $i = 0,\dots,N_1$.
Applying the inner dynamics, we get $(I_{\text{N}_1 + 1},\theta_{\text{N}_1 + 1}) = \phi_{t_{N_1}}(I_{\text{N}_1},\theta_{\text{N}_1} )$ with $\theta_{\text{N}_1 + 1}\in (\theta_{\text{A}}(\psi_2(I_{N_1})),2\pi)$ and then we construct a pseudo-orbit $\{(I_i,\theta_i): i = N_1 +1,...,N_1+M_1\}\subset \mathcal{L}_{\M,\text{A}}^*(I_{\text{N}_1+1},\theta_{\text{N}_1+1}) = l_{\text{N}_1 +1}$ with $\theta_{i}\in(\theta_{\text{N}_1 +1},2\pi)$, $2\pi -\theta_{\text{N}_1 + \M_{1}} = \mathcal{O}(\varepsilon^2)$.
Applying the inner dynamics, we get  $(I_{\text{N}_1 + \M_1+1}, \theta_{\text{N}_1 + \M_1+1})=\phi_{t_{\text{N}_1 + \M_1}}(I_{\text{N}_1 + \M_1}, \theta_{\text{N}_1 + \M_1})$ with $\theta_{\text{N}_1 + \M_1 +1}\in \theta_{\text{A}}(\psi_{2}(I_{\text{N}_1 + \M_1}),2\pi))$.
Recursively, we construct pseudo-orbit $\{(I_i,\theta_i): i = \text{N}_1 +1,...,\text{N}_2\}$ such that $I_{\text{N}_2}\geq I_{++}.$ We finally follow the highway from $I_{++}$ to $I^*$ constructing a pseudo-orbit $\{I_i,\theta_i): i = \text{N}_2, ...,I_{\text{N}_3}\}\subset H_{\text{r}}$ with $I_{\text{N}_3} = I^*$.

Using the symmetry properties (see Lemma \ref{lem:geometrical_lemmas}) introducing $I_i = -I_i$ for $i<0$ we have a pseudo-orbit $\{(I_i,\theta_i): \left|i\right|\leq \text{N}_3\}$ with $I_{-\text{N}_3} = -I^*$, $I_{\text{N}_3}=I^*$.
Using now the same shadowing techniques as in the proof of Theorem \ref{teo:main_theorem}, there exists a diffusion trajectory.
If $a_{10}<0$, changing $H_{\text{r}}$ to $H_{\text{l}}$ all the previous reasoning applies. 
\end{proof}

\begin{remark}
For the proof of this theorem we have  chosen a simple pseudo-orbit, just choosing the scattering map $\mathcal{S}_{\M,A}$ when it was not unique.
Of course, there is a lot of freedom in choosing pseudo-orbits, and we do not claim that the one chosen here is the best one concerning minimal time of diffusion. 
\end{remark}

\begin{figure}[h]
\centering
\subfigure[$\mathcal{S}_{\M}$ for $\mu=1.5$ ]{\includegraphics[scale=0.27]{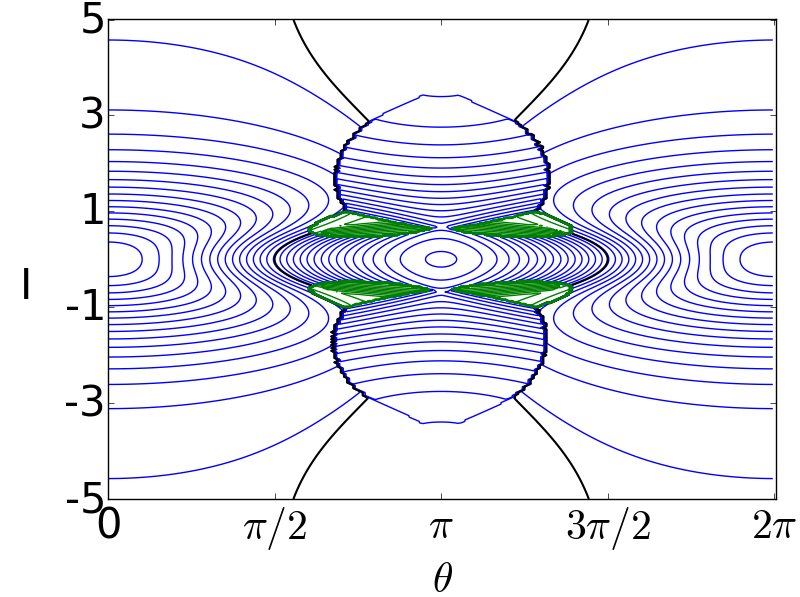}}
\qquad
\subfigure[$\mathcal{S}_{\M}$ combined with inner map (in red)]{\includegraphics[scale=0.27]{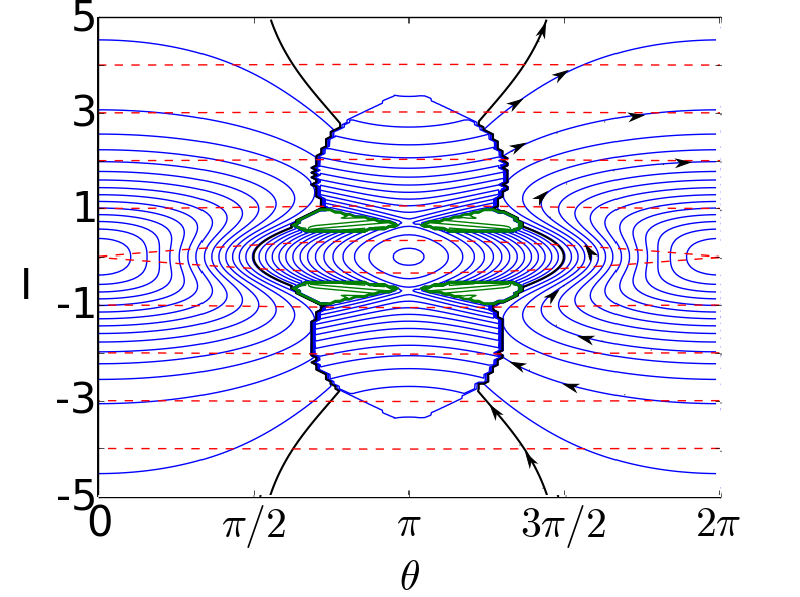}}
\caption{For $\mu = 1.5$, highway are not preserved. Inner map and scattering map can be adequately combined}
\end{figure}
\begin{remark}
\textbf{A rough estimate for $\varepsilon^* = \varepsilon^*(I^*)$ of Theorem \ref{teo:main_theorem} }.
The scattering map $\mathcal{S}(I,\theta)$ \eqref{eq:second_definition_SM_int} is the $-\varepsilon$ time map of the Hamiltonian $\mathcal{L}^*(I,\theta)$ given in \eqref{eq:red_poi_fun_theta}, up to order $\mathcal{O}(\varepsilon^2)$.
Therefore, as already noticed in Remark \ref{rem:norm_red_poin_fun}, if $\left|\partial \mathcal{L}^*/\partial \theta (I,\theta)\right|\leq \varepsilon$ or $\left|\partial \mathcal{L}^*/\partial I (I,\theta)\right|\leq \varepsilon$, the level curves of $\mathcal{L}^*(I,\theta)$ are not useful enough to describe the orbits of $\mathcal{S}$.
It is easy to check that $\nabla\mathcal{L}^*(I,\theta)$ only vanishes for $I= 0$, $\theta = 0,\,\pi \mod 2\pi$ and that $\left\|\nabla\mathcal{L}^*(I,\theta)\right\|\lesssim 8\pi \left|a_{10}I\right|e^{-\pi\left|I\right|/2}$ for $\left|I\right|\rightarrow +\infty$.
Thus, in general one has to avoid small neighborhoods of $(I,\theta) = (0,0),\,(0,\pi)$ and take care in regions where $\left|I\right|$ is very large.
In particular, the highways $H_{\text{l}},\,H_{\text{r}}$ are far from $(I,\theta) = (0,0),\,(0,\pi)$ and on them  $\left\|\nabla\mathcal{L}^*(I,\theta)\right\|\geq A_{10}(I)(1- \mathcal{O}(\beta(I)\mu))\gtrsim 4\pi\left|a_{10}I \right|e^{-\pi\left|I\right|/2}$ for large $\left|I\right|$, from which we get an upper bound for $\varepsilon^*(I^*)$, which is exponentially small in $\left|I^*\right|$ for  large $\left|I^*\right|:$
$$\varepsilon^*(I^*)<4\pi\left|a_{10}\right|\left|I^*\right|\exp(-\pi\left|I^*\right|/2).$$ 
For smaller values of $I^*$, one can compute numerically the level curves of $\left\|\nabla\mathcal{L}^*(I,\theta)\right\| = \varepsilon$ and obtain $\varepsilon^* >\varepsilon^*(I^*)$ such that $\left\|\nabla\mathcal{L}^*(I,\theta)\right\| = \varepsilon^*$ implies $\left|I\right|>\left|I^*\right|$.
See Table \ref{tab:estimates_09} for some values of $I^*$, and $\mu = 0.9$.
\end{remark}
\begin{table}[h]
\centering
\begin{tabular}{@{}cclll@{}}
\toprule
\multicolumn{1}{r}{}$I^*$ & 1  & 2  & 3 & 4 \\ \midrule
                     $\varepsilon^*(I^*)$ & 1.4 & 0.75 & 0.25 & 0.07 \\ \bottomrule
\end{tabular}
\caption{Estimates of $\varepsilon^*$ for $\mu = 0.9$}\label{tab:estimates_09}
\end{table}


\section{The time of diffusion}\label{sec:time of diffusion}

In this section we will provide an estimate of the diffusion time.
For simplicity, we are going to estimate the time for a diffusion using a highway (see Definition \ref{def:highways}) as a guide, that is, we are going to construct a pseudo-orbit close a the highway.
This implies to iterate the scattering map using as  initial point a point on a highway.
As we have seen before, see Subsection \ref{sub:Meln pot and crests}, one iterate of $\mathcal{S}_{\M}(I,\theta)$ is approximated by $-\varepsilon$ time map of the Hamiltonian $\mathcal{L}^*_{\M}(I,\theta)$ up to $\mathcal{O}(\varepsilon^2)$.
However, if we iterate the scattering map a number $n$ of times, it generates a propagated error with respect to the level curve of $\mathcal{L}_{\M}^*(I,\theta)$.

So, first we study the error generated by $n$ iterates of the scattering map.
Later, we will estimate the time of diffusion along the highway combining the scattering and the inner maps.  

\subsection{Accuracy of the scattering map}\label{sec:accuracy}

Equation \eqref{eq:second_definition_SM_int} for the scattering map $\mathcal{S}$ is good enough up to an error of $\mathcal{O}(\varepsilon^2)$ for understanding one iterate of $\mathcal{S}$.
But if we consider $\mathcal{S}^n$, that is, $n$-iterates of $\mathcal{S}$ some problems appear.
These problems are related with the lack of precision of the equation \eqref{eq:second_definition_SM_int}:
\begin{itemize}
\item Equation \eqref{eq:second_definition_SM_int} of the scattering map has a relative error of order $\mathcal{O}(\varepsilon)$ and an absolute error $\mathcal{O}(\varepsilon^2)$.
Therefore, for $n$-iterates, when $n$ is large, the error is propagated in a such way that it cannot be discarded.

\item Highways are unstable, i.e., the nearby level curves of $\mathcal{L}^*$ move away from  highways (see instance Fig.\ref{fig:crest_and_SM_mu_06}.b).
\end{itemize}

Now, our goal is to show how we can control these errors along a region $U$ in the phase space $(I,\theta)$ close to a highway.
Basically, the control is to choose a good moment and interval to apply the inner map to come back to the highway and to maintain the errors small enough.

\paragraph{The propagated error\\\\}
After iterating $n$  times formula \eqref{eq:second_definition_SM_int} for the scattering map, one gets for $(I_n,\theta_n) = \mathcal{S}^n(I_0,\theta_0)$:
\begin{equation}
I_{n} = I_{0} + \varepsilon\sum_{j = 0}^{n-1}\frac{\partial \mathcal{L}^*}{\partial \theta}(I_j,\theta_j) + \mathcal{O}(n\varepsilon^2),\quad\text{and also}\quad \theta_{n} = \theta_0 - \varepsilon\sum_{j = 0}^{n-1}\frac{\partial \mathcal{L}^*}{\partial I}(I_j,\theta_j) + \mathcal{O}(n\varepsilon^2).\label{eq:I_n_and_theta_n}
\end{equation}

From now on, in this section, we will use the following notation:
\begin{itemize}
\item $\mathcal{S}(I,\theta)$ is the scattering map, see \eqref{eq:second_definition_SM_int}.
\item $S_{\text{T}}(I,\theta) = \left(I+\varepsilon\,\partial \mathcal{L}^{*}/\partial \theta(I,\theta),\theta-\varepsilon \, \partial \mathcal{L}^{*}/\partial I(I,\theta)\right)$ is the truncated scattering map.
\item $S_{0,t}(I,\theta) = (I(t),\theta(t))$ is the solution of the Hamiltonian system
\begin{eqnarray}\label{eq:Edo}
\dot{I}(t) = \frac{\partial \mathcal{L}^*}{\partial \theta}(I(t),\theta(t)) &\quad & \dot{\theta}(t) = -\frac{\partial \mathcal{L}^*}{\partial I}(I(t),\theta(t)),
\end{eqnarray}
with initial condition $(I(0),\theta(0)) = (I,\theta)$.
\end{itemize}

Let $(I_{\h},\theta_{\h})$ be a point in the highway.
The error between the scattering map and the level curve of the reduced Poincar\'{e} function after $n$-iterates is given by 
\begin{equation}
\left\|\mathcal{S}^{n}(I_{\h} + \Delta I,\theta_{\h} + \Delta\theta) - S_{0,n\varepsilon}(I_{\h},\theta_{\h})\right\|,\label{eq:origem_error}
\end{equation}
where $\Delta I$ and $\Delta \theta$ are small.
Note that we can rewrite  \eqref{eq:origem_error} as
\begin{eqnarray*}
\|(\mathcal{S}^{n}(I_{\h}+\Delta I,\theta_{\h}+ \Delta\theta) &-& S^n_{\text{T}}(I_{\h}+\Delta I,\theta_{\h}+ \Delta\theta))\\
+(S^n_{\text{T}}(I_{\h}+\Delta I,\theta_{\h}+ \Delta\theta)&-&S_{0,n\varepsilon}(I_{\h}+\Delta I,\theta_{h}+\Delta \theta))\\ 
+ S_{0,n\varepsilon}(I_{\h}+\Delta I,\theta_{h}+\Delta \theta))&-&S_{0,n\varepsilon}(I_{\h},\theta_{\h}))\|.
\end{eqnarray*}

We now proceed to study each subtraction.
\begin{itemize}
\item We begin with $\mathcal{S}^{n}(I_{\h}+\Delta I,\theta_{\h}+ \Delta\theta) - S^n_{\text{T}}(I_{\h}+\Delta I,\theta_{\h}+ \Delta\theta)$.
From \eqref{eq:I_n_and_theta_n}, we can readily obtain by induction that
\begin{equation}\label{eq:inequality_1}
\mathcal{S}^{n}(I_{\h}+\Delta I,\theta_{\h}+ \Delta\theta) - S^n_{\text{T}}(I_{\h}+\Delta I,\theta_{\h}+ \Delta\theta) =\mathcal{O}(n\varepsilon^2).
\end{equation}
\item Now we consider $S^n_{\text{T}}(I_{\h}+\Delta I,\theta_{\h}+ \Delta\theta)-S_{0,n\varepsilon}(I_{\h}+\Delta I,\theta_{h}+\Delta \theta)$.
By the definition of $S_{\text{T}}$ we have that $S^n_{\text{T}}$ is the $n$-step of the Euler method with step size $\varepsilon$ in each coordinate for solving the system \eqref{eq:Edo}.
It is not difficult to check the standard bound (see, for instance, \cite{StoerB02}) 
\begin{equation}\label{eq:inequality_2}
\left\|S^n_{\text{T}}(I_{\h}+\Delta I,\theta_{\h}+ \Delta\theta)-S_{0,n\varepsilon}(I_{\h}+\Delta I,\theta_{h}+\Delta \theta)\right\| \leq \frac{L\varepsilon}{2}\left[(1+\varepsilon K)^n - 1\right],
\end{equation}
where $M:=\max_{(I,\theta)\in U} \left\|\J\textbf{H}(I,\theta)\left(\J\nabla\mathcal{L}^*(I,\theta\right))^T\right\| \text{ and } L = \max_{(I,\theta)\in U} \left\|\nabla\mathcal{L}^*(I,\theta)\right\|.$
\item Now we look for the last subtraction 
$ S_{0,n\varepsilon}(I_{\h}+\Delta I,\theta_{h}+\Delta \theta))-S_{0,n\varepsilon}(I_{\h},\theta_{\h})$.
Applying Gr\"{o}nwall's inequality on the variational equation associated to the Hamiltonian vector field $-\nabla\mathcal{L}^*(I,\theta)$, one gets
\begin{equation}\label{eq:inequality_3}
\left\| S_{0,\varepsilon n}(I_{\h}+\Delta I,\theta_{h}+\Delta \theta))-S_{0,\varepsilon n}(I_{\h},\theta_{\h})\right\| \leq \left\|(\Delta I,\Delta \theta)\right\|e^{K\varepsilon n}.
\end{equation}
\end{itemize}

We can now conclude from \eqref{eq:inequality_1}, \eqref{eq:inequality_2} and \eqref{eq:inequality_3}, that the propagated error is 

\begin{equation*}
\left\|\mathcal{S}^{n}(I_{\h} + \Delta I,\theta_{\h} + \Delta\theta) - S_{0,n\varepsilon}(I_{\h},\theta_{\h})\right\| \leq \mathcal{O}(n\varepsilon^2)+\frac{L\varepsilon}{2}\left[(1+\varepsilon K)^n - 1\right] + \left\|(\Delta I,\Delta \theta)\right\|e^{\mathcal{K}\varepsilon n}
\end{equation*}

To avoid large propagated errors, one has to choose $n$ such that $n\varepsilon \ll 1$.
For instance, taking 
\begin{equation}\label{eq:def_n_iterates}
n = \varepsilon^{-c},
\end{equation}
with $0<c<1$ (which implies $n\varepsilon \ll 1$) and $\left\|(\Delta I,\Delta \theta)\right\|  = \varepsilon^a$, $a>0$, one gets
\begin{equation}\label{eq:error_final}
\left\|\mathcal{S}^{n}(I_{\h} + \Delta I,\theta_{\h} + \Delta\theta) - S_{0,n\varepsilon}(I_{\h},\theta_{\h})\right\| = \mathcal{O}(\varepsilon^{2-c}, \varepsilon^a).
\end{equation}

\subsection{Estimate for the time of diffusion}\label{sec:time of diffusion}

In this section our goal is to estimate the time of diffusion along the highway.
We have three different types of estimates associated to the time of diffusion.

\begin{itemize}
\item The total number of iterates $N_{\s}$ of the scattering map.
This is the number of iterates that scattering map spends to cover a piece of a level curve of the reduced Poincar\'{e} function $\mathcal{L}^*$.
\item The time under the flow along the homoclinic invariant manifolds of $\widetilde{\Lambda}$.
This is the time spent by each application of the scattering map following the concrete homoclinic orbit to $\widetilde{\Lambda}$ up to a distance $\delta$ of $\widetilde{\Lambda}$.
This time is denoted by $T_{\h} = T_{\h}(\delta)$.
\item The time under the inner map.
This time appears if we use the inner map between iterates of the scattering map (it is sometimes called ergodization time) and we denoted it by $T_{\text{i}}$. 
\end{itemize}

For each iterate of the scattering map we have to consider the time $T_{\h}$.
Besides, we have seen in the previous subsection that to control the propagated error, we iterate \emph{successively} the scattering map just a number $n = \varepsilon^{-c}$ of times, $0<c\ll1$.
From now on we denote this number $n$ by $N_{\s\s}$.
So, after $N_{\s\s}$ iterates of the scattering maps we apply the inner dynamics during some time $T_{\text{i}}$ to come back to a distance $\varepsilon^{a}$ to the highway.
Therefore, the total time spent under the inner map is $\left\lfloor N_{\s}/N_{\s\s}\right\rfloor T_{\text{i}}$.
We estimate that the diffusion time along the highway is thus
\begin{equation}\label{eq:diffusion_time}
T_{\text{d}} = N_{\s} T_{\text{h}} +\left\lfloor N_{\s}/N_{\s\s}\right\rfloor T_{\text{i}}.
\end{equation} 

\begin{theorem}\label{prop:time}
The time of diffusion $T_{\text{d}}$ close to a highway of Hamiltonian \eqref{eq:hamiltonian_int}$+$\eqref{eq:perturbation_int} between $-I^*$ to $I^*$, for any $0<I^*<I_{+}$, with $I_{+}$ given in Proposition \ref{pro:geometrical_proposition}, satisfies the following asymptotic expression
\begin{equation*}
T_{\text{d}} = \frac{T_{\text{s}}}{\varepsilon}\left[2\log\left(\frac{C}{\varepsilon}\right) + \mathcal{O}(\varepsilon^{b})\right],\text{ for }\varepsilon\rightarrow 0,\text{ where }0<b<1,
\end{equation*}
with
\begin{equation*}
T_{\text{s}} = T_{\text{s}}(I^*,a_{10},a_{01}) = \int_{0}^{I^*}\frac{-\sinh(\pi I/2)}{\pi a_{10}I\sin\psi_{\h}(I)}dI,
\end{equation*}
where $\psi_{\text{h}}(I)$ is the parameterization \eqref{def:psi} of the highway $\mathcal{L}^*(I,\psi_{\h}) = A_{00} + A_{01}$, and
\begin{equation*}
C = C(I^*,a_{10},a_{01}) = 16\left|a_{10}\right|\left(1 + \frac{1.465}{\sqrt{1- \mu^2A^2}}\right)
\end{equation*}
where $A = \max_{I\in\left[0,I^*\right]} \alpha(I)$, with $\alpha(I)$ given in \eqref{eq:alpha} and $\mu = a_{10}/a_{01}$. 
\end{theorem}
The proof of this Proposition is a consequence of the following four subsections.

\subsubsection{Number of iterates $N_{\text{s}}$ of the scattering map}

The scattering map $(I',\theta') = \mathcal{S}(I,\theta)$ given in \eqref{eq:second_definition_SM_int} can be rewritten as
\begin{eqnarray*}
\frac{I'-I}{\varepsilon} = \frac{\partial \mathcal{L}^*}{\partial \theta}(I,\theta) + \mathcal{O}(\varepsilon) &\quad&\frac{\theta'-\theta}{\varepsilon} = -\frac{\partial \mathcal{L}^*}{\partial I}(I,\theta) + \mathcal{O}(\varepsilon).
\end{eqnarray*}

Hence, disregarding the $\mathcal{O}(\varepsilon)$ terms, we define 
\begin{eqnarray}
\frac{d I}{d\upsilon} = \frac{\partial \mathcal{L}^*}{\partial \theta}(I,\theta) &\quad & \frac{d\theta}{d\upsilon} = -\frac{\partial \mathcal{L}^*}{\partial I}(I,\theta),\label{eq:diff_equation_t_scat_map}
\end{eqnarray}
where $\upsilon$ is a new parameter of time. Note that $\mathcal{L}^*(I,\theta)$ is a first integral of (\ref{eq:diff_equation_t_scat_map}) and that the highway has the equation $\mathcal{L}^*(I,\theta) = A_{00} +A_{01}$.
Recalling formula \eqref{eq:der_red_poi_fun_theta} for $\partial \mathcal{L}^*/\partial \theta (I,\theta)$, the equation for $I$ reads as 
\begin{equation*}
\frac{d I}{d\upsilon} = \frac{\partial\mathcal{L}^*}{\partial \theta}(I,\theta) = -A_{10}(I)\sin\psi,
\end{equation*}
where $\psi = \theta - I\tau^*(I,\theta)$ as given in (\ref{def:psi}).
We choose the highway $H_{\text{r}}$ for $a_{10}>0$ (or $H_{\l}$ for $a_{01}<0$) to ensure that $\partial \mathcal{L}^*/\partial \theta (I,\theta)>0$ (see Definition \ref{def:highways}). 
This implies that we can rewrite the equation above as 
$$\frac{d\upsilon}{dI} = \frac{-1}{A_{10}(I)\sin\psi_{\text{h}}}$$
so that
\begin{equation*}
T_{\s} := \upsilon  = \int_{I_{0}}^{I_{\text{f}}}\frac{-1}{A_{10}(I)\sin\psi_{\text{h}}(I)}dI
= \int_{I_{0}}^{I_{\text{f}}}\frac{-\sinh(\pi I/2)}{2\pi I a_{10}\sin\psi_{\text{h}}(I)}dI\label{eq:time_T_s}
\end{equation*}
is the time of diffusion in the interval $[I_{0},I_{f}]$ of values of $I$ following the flow
\eqref{eq:diff_equation_t_scat_map}.
\begin{remark} If we consider an interval of diffusion as in Theorem \ref{teo:main_theorem}, that is, $\left[-I^*,I^*\right]$, the time $T_{\s}$ is
$$T_{\text{s}} = \int_{0}^{I^*}\frac{-\sinh(\pi I/2)}{\pi I a_{10}\sin\psi_{\text{h}}(I)}dI.$$
\end{remark}

\begin{remark}
Observe that 
\begin{equation*}
T_{\s} \geq \frac{1}{2\pi a_{10}}\left(\text{Shi}\left(\frac{I_{f}\pi}{2}\right)-\text{Shi}\left(\frac{I_{0}\pi}{2}\right)\right),\label{eq:lower_bound}
\end{equation*}
where the function Shi$(x)$ is defined as 
$$\text{Shi}(x) := \int_{0}^{x}\frac{\sinh \sigma}{\sigma}d\sigma.$$
\end{remark}

The time $T_{\s}$ has been computed from the continuous dynamics (\ref{eq:diff_equation_t_scat_map}) . 
But the scattering map generates a discrete dynamics with a $\varepsilon$-step.
Then for us, the important information is the number of iterations of the scattering map \eqref{eq:second_definition_SM_int} from $I_0$ to $I_{\text{f}}$ which is given by
\begin{equation*}\label{eq:num_ite_sct_map}
N_\text{s}= \frac{T_{\text{s}}}{\varepsilon}(1 + \mathcal{O}(\varepsilon)).
\end{equation*}

\subsubsection{Time of the travel $T_{\text{h}}$ on the invariant manifold}\label{sec:time_invariant_manifold}

Let $\tilde{x}_-$ and $\tilde{x}_+$ be on $\tilde{\Lambda}$ such that $S(\tilde{x}_-)=\tilde{x}_+$.
We now estimate the time of the flow from a point $\delta$-close to $ \tilde{x}_-$ to a  point $\delta$-close to $\tilde{x}_+$. 

Recall that the unperturbed separatrices \eqref{eq:separatrices} are given by $(p_{0}(t),q_{0}(t)) = \left(2/\cosh t,4\arctan e^{ t}\right)$.
We have $(p_{\varepsilon}(\tau),q_{\varepsilon}(\tau))= \left(2/\cosh \tau,4\arctan e^{ \tau}\right) + \mathcal{O}(\varepsilon)$, where  $(p_{\varepsilon}(\tau),q_{\varepsilon}(\tau))\in\overline{\text{B}_{\delta}(0)}\cap W_{\varepsilon}^{s,u}(0)$.

Note that when $\tau \rightarrow \pm \infty$,
\begin{equation*}\label{eq:p_0_infinito}
p_{0}(\tau) = \frac{4}{e^{\left|\tau\right|}}\left(1 - e^{-2\left|\tau\right|}+ e^{-4\left|\tau\right|}+\dots\right) = \frac{4}{e^{\left|\tau\right|}}\left(1 + \mathcal{O}(e^{-2\left|\tau\right|})\right).
\end{equation*}
Besides, as $\dot{q}_{0}(\tau)=\partial H_{0}/\partial p = p_{0}(\tau)$,
we also have
\begin{equation*}
q_{0}(\tau) = \mp\frac{4}{e^{\left|\tau\right|}}\left(1 +\mathcal{O}(e^{-2\left|\tau\right|})\right)\mod 2\pi\quad \text{when  }  \tau \rightarrow \pm\infty.
\end{equation*}
We consider starting and ending points on $\partial \text{B}_{\delta}(0,0)$.
Then, denoting by $\tau_{\text{f}} = -\tau_{\text{i}}$ the initial and final points, we have
\begin{equation*}
q_{0}^2(\tau_\text{i}) + p_{0}^2(\tau_\text{i}) =q_{0}^2(\tau_\text{f}) + p_{0}^2(\tau_\text{f})  = \left[\frac{4}{e^{u}}\left(1 + \mathcal{O}(e^{-2u})\right)\right]^2 + \left[-\frac{4}{e^u}\left(1 + \mathcal{O}(e^{-2u})\right)\right]^2 = \delta^2,
\end{equation*}
where $u = \left|\tau_{\text{i}}\right|,\, \left|\tau_{\text{f}}\right|$.
Therefore,
\begin{equation}
\frac{4\sqrt{2}}{e^{u}}\left(1 + \mathcal{O}(e^{-2u})\right) = \delta.\label{eq:delta_e_s_first}
\end{equation}

Note that by the above equation $\delta = \mathcal{O}(e^{-u})$, thus $e^{-2u} = \mathcal{O}(\delta^2)$.
Hence, we can rewrite equation \eqref{eq:delta_e_s_first} as
\begin{equation}\label{eq:delta_as_a_function_of_e_tf}
e^{u} = \frac{4\sqrt{2}}{\delta}\left(1 + \mathcal{O}(\delta^2)\right). 
\end{equation}
So,
\begin{equation*}
u = \log\left[\frac{4\sqrt{2}}{\delta}\left(1 + \mathcal{O}(\delta^2)\right)\right] = \log\left(\frac{4\sqrt{2}}{\delta}\right) + \mathcal{O}(\delta^2).
\end{equation*}
Since $\Delta \tau = 2u$, we finally have
\begin{equation}\label{eq:time_h_as_a_funcrion_of_delta}
T_{\h} =  2\log\left(\frac{4\sqrt{2}}{\delta}\right) + \mathcal{O}(\delta^2) + \mathcal{O}(\varepsilon).
\end{equation} 

It is now necessary to estimate a value for $\delta$ and we want $\delta$ small enough such that this choice does not affect significantly the scattering map \eqref{eq:second_definition_SM_int}, that is, that the level curves of the reduced Poincar\'{e} remain at a distance of $\mathcal{O}(\varepsilon)$. 
From Proposition \ref{prop:melnpot} the Melnikov potential, using that $p_0^2/2+ \cos q_0 - 1 = 0$, is 
$$ \mathcal{L}(I,\varphi,s) = \frac{1}{2} \int_{-\infty}^{+\infty}p_0^2(\sigma)\left(a_{00} + a_{10}\cos(\varphi + I\sigma) + a_{01}\cos(s + \sigma)\right)d\sigma.$$

The reduced Poincar\'{e} function \eqref{eq:red_poi_func_1} $\mathcal{L}^*(I,\theta)$ is 
$$\mathcal{L}^*(I,\theta) = \frac{1}{2}\int_{-\infty}^{+\infty}p_0^{2}(\sigma)\left(a_{00} + a_{10}\cos(\varphi + I(\sigma - \tau^*(I,\varphi,s))) + a_{01}\cos(s - \tau^*(I,\varphi,s) + \sigma)\right)d\sigma.$$

Considering the diffusion along the highways, recall that $\psi$, given in (\ref{def:psi_original}), is well defined and, as in  \eqref{eq:red_poin_fun_psi_2}, we can write the reduced Poincar\'{e} function on the variables $(I,\psi)$ as
\begin{eqnarray*}
\mathfrak{L}^*(I,\psi) &=& \frac{1}{2}\int_{-\infty}^{+\infty}p_0^{2}(\sigma)\left(a_{00} + a_{10}\cos(\psi + I\sigma) + a_{01}\cos(\xi(I,\psi) + \sigma)\right)d\sigma \\
&=&A_{00} + A_{10}(I)\cos\varphi + A_{01}\cos\xi(I,\varphi).
\end{eqnarray*}

As we want to preserve the level curves of the reduced Poincar\'{e} function up to $\mathcal{O}(\varepsilon)$, we need $t_{\text{i}}$ and $t_{\text{f}}$ such that the integration above along all the real numbers does not change much when the interval of integration is $\left[t_{\text{i}},t_{\text{f}}\right]$, more precisely, given a $\varepsilon>0$

\begin{equation}\label{eq:inequality}
\left|\frac{\partial \mathfrak{L}^*}{\partial I}(I,\psi) - \left(\frac{\partial \mathfrak{L}^*}{\partial I}(I,\psi)\right)_{\delta}\right| <\varepsilon \quad\quad \text{ and }\quad\quad \left|\frac{\partial \mathfrak{L}^*}{\partial \psi}(I,\psi) - \left(\frac{\partial \mathfrak{L}^*}{\partial \psi}(I,\psi)\right)_{\delta}\right| <\varepsilon, 
\end{equation}
where $\left(\frac{\partial \mathfrak{L}^*}{\partial \gamma}(I,\psi)\right)_{\delta}$ is given, for  $\gamma \in \{\psi,I\}$ by
$$\frac{-1}{2}\int_{t_\text{i}}^{t_{\text{f}}}\frac{\partial}{\partial \gamma}\left(p_0^{2}(\sigma)\left(a_{00} + a_{10}\cos(\varphi + I(\sigma - \tau^*(I,\varphi,s))) + a_{01}\cos(s - \tau^*(I,\varphi,s) + \sigma)\right)\right)d\sigma.$$

Using that $\left|\alpha'(I)\right| <1.465$, one computes that
$$\left|\frac{\partial \mathfrak{L}^*}{\partial \gamma}(I,\psi) - \left(\frac{\partial \mathfrak{L}^*}{\partial \gamma}(I,\psi)\right)_{\delta}\right| < Ce^{-t_{\text{f}}} \quad \gamma\in\{\psi, I\},$$
where $C = 16\left(\left|a_{10}\right|+1.465\left|a_{01}\right|\left|\mu\right|/\sqrt{1-\mu^2A^2}\right)$, $A = \max_{I\in[0,I^*]}\alpha(I)$. 
By \eqref{eq:delta_as_a_function_of_e_tf} with $u = t_{\text{f}}$, this is equivalent to
$$\left|\frac{\partial \mathfrak{L}^*}{\partial \gamma}(I,\psi) - \left(\frac{\partial \mathfrak{L}^*}{\partial \gamma}(I,\psi)\right)_{\delta}\right| < \frac{C\delta(1 + \mathcal{O}(\delta^2))}{4\sqrt{2}}.
$$

To satisfy Eq.\eqref{eq:inequality} we have to take a $\delta$ such that the above right hand side is less or equal than $\varepsilon$.
For simplicity, we take $\delta$ satisfying the equality, that is,
$$
\delta = \frac{4\sqrt{2}\varepsilon}{C}(1 +\mathcal{O}(\varepsilon_{0}^{2})).
$$
Inserting this value of $\delta$ in \eqref{eq:time_h_as_a_funcrion_of_delta}, we can conclude that 
$$T_{\text{h}} = 2\log\left(\frac{16\left|a_{10}\right|\left(1+\frac{1.465}{\sqrt{1-\mu^2A^2}}\right)}{\varepsilon}\right) +\mathcal{O}(\varepsilon).
$$

\subsubsection{Time $T_{\text{i}}$ under the inner map \label{sec:time_inner}}

To build of the pseudo-orbit which shadows the real diffusion orbit, we need, after each $N_{\s\s}$-iterates of the scattering map ($N_{\s\s} = \left\lceil\varepsilon^{-c}\right\rceil$, see \eqref{eq:def_n_iterates}) , to apply the inner flow to return to the same level curve of $\mathcal{L}^*$ (or close enough).
The time spent by the inner flow is the time $T_{\text{i}}$, which we are going to estimate.

Recall that $\widetilde{\Lambda}_{\varepsilon} = \widetilde{\Lambda}$, where $\widetilde{\Lambda}$ is a NHIM of the unperturbed case (see Section \ref{sec:the_system}).
We will calculate the time for the flow of the unperturbed case because in our case it is a good approximation, that is, along \NH\, lines $(I,\varphi + It, s + t)$ (see Section \ref{sec:the_system}).

Given $\varepsilon>0$ small enough, our goal is to calculate $t>0$ such that 
\begin{equation}\label{eq:inner_time}
\left|(I,\varphi + It,s+t) - (I,\varphi,s)\right|<\varepsilon^a,
\end{equation}
that is, $\left|I(2\pi k) - 2\pi l\right| <\varepsilon^a$ for some integer $k$, $l$, or equivalently
\begin{equation}\label{eq:ergodization_time}
\left|I - \frac{l}{k}\right|<\frac{\varepsilon^{a}}{2\pi k}.
\end{equation}
We now recall the Dirichlet Box Principle:
\begin{proposition}{\textbf{(Dirichlet Box Principle)}}
Let $N$ be a positive integer and let $\alpha$ be any real number.
Then there exists positive integers $k\leq N$ and $l\leq \alpha N$ such that
$$\left|\alpha - \frac{l}{k}\right| \leq \frac{1}{k(N+1)}.$$ 
\end{proposition}

Define $N :=\lceil 2\pi/\varepsilon^a-1\rceil $, the smaller natural number such that it is greater or equal than $2\pi/\varepsilon^a-1$.
Then from the Dirichlet Box Principle, there exist $k, l$ satisfying the condition (\ref{eq:ergodization_time}) such that
$k \leq N$ and $l\leq I N$.
Then $T_{\text{i}}= 2\pi k$ is the time required for \eqref{eq:inner_time}, called the ergodization time.
Note that for any $\varphi$, 
\begin{equation}\label{eq:upper_bound_inner_time}
T_{\text{i}} \leq 2\pi N = 2\pi\left\lceil\frac{2\pi}{\varepsilon^a}-1\right\rceil,
\end{equation}

So that $T_{\text{i}} = \mathcal{O}(\varepsilon^{-a})$.

\subsubsection{Dominant time and the order of diffusion time}

We finally put together the estimates of $N_{\s},\,T_{\text{h}}$ and $T_{\text{i}}$, jointly with $N_{\s\s} = \varepsilon^{-c}$ in the formula for the time of diffusion \eqref{eq:diffusion_time}.
If we look just at the order of the time of diffusion we have
\begin{equation*}
T_{\text{d}} = N_{\s}T_{\h}+ \left\lfloor N_{\s}/N_{\s\s}\right\rfloor T_{\text{i}}  =\mathcal{O}(\varepsilon^{-1}\log\varepsilon^{-1}) + \mathcal{O}(\varepsilon^{c-a-1}).
\end{equation*}

Choosing $0<a<c$ the term containing the time $T_{\text{i}}$ under the inner map is negligible compared with the term containing the time of travel $T_{\h}$ along the homoclinic orbit: $\varepsilon^{c-a-1}\ll (1/\varepsilon)\log 1/\varepsilon$.
We finally obtain the desired estimate for the time of diffusion
\begin{equation*}
T_{\text{d}} = \frac{T_{\text{s}}}{\varepsilon}\left[2 \log \frac{C}{\varepsilon} + \mathcal{O}(\varepsilon^b)\right],
\end{equation*}
where $b = c-a$.
Since $c<1$, $0<b<1$.
Notice that by the choice of the parameter $0<a<c\ll 1$, the accuracy of the scattering map given in \eqref{eq:error_final} is $\mathcal{O}(\varepsilon^a)$.
\bibliography{references}

\end{document}